\documentclass{amsart}
\usepackage[a4paper,outer=3.4cm,inner=3.4cm,bottom=4cm,top=3.6cm]{geometry}

\usepackage[english]{babel}
\usepackage[utf8]{inputenc}
\usepackage{amsmath,amsthm,amssymb}
\usepackage{enumerate}

\usepackage{tikz-cd}
\usepackage{verbatim}
\usepackage[shortlabels]{enumitem}
\setlist[enumerate]{label=\arabic*.}
\setlist{leftmargin=10mm}

\usepackage{mathtools}
\usepackage{latexsym}
\usepackage{thmtools}
\usepackage{array}
\usepackage[colorlinks=true]{hyperref}
\hypersetup{colorlinks, citecolor=blue, filecolor=black, linkcolor=black, urlcolor=black}
\usepackage{stmaryrd}

\usepackage[capitalize]{cleveref}
\crefformat{equation}{\normalfont(#2#1#3)}

\numberwithin{equation}{section}
\newtheorem{Theorem}[equation]{Theorem}
\newtheorem{Lemma}[equation]{Lemma}
\newtheorem{Proposition}[equation]{Proposition}
\newtheorem{Corollary}[equation]{Corollary}

\theoremstyle{definition}
\newtheorem{Definition}[equation]{Definition}

\newtheorem{Remark}[equation]{Remark}
\newtheorem{Example}[equation]{Example}
\newtheorem{Claim}[equation]{Claim}

\newcommand{\dmd}{\diamondsuit}
\newcommand{\im}{\operatorname{im}}
\newcommand{\wh}{\widehat}

\newcommand{\N}{\mathbb{N}}
\newcommand{\Z}{\mathbb{Z}}
\newcommand{\Q}{\mathbb{Q}}

\newcommand{\F}{\mathbb{F}}

\newcommand{\KS}{\mathrm{KS}}

\renewcommand{\O}{\mathcal{O}}

\newcommand{\plim}{\textstyle \varprojlim}

\newcommand{\coker}{\operatorname{coker}}

\newcommand{\Spec}{\operatorname{Spec}}
\newcommand{\Spf}{\operatorname{Spf}}
\newcommand{\Spa}{\operatorname{Spa}}
\newcommand{\id}{{\operatorname{id}}}

\newcommand{\an}{{\mathrm{an}}}
\newcommand{\dR}{\mathrm{dR}}

\newcommand{\hotimes}{\widehat{\otimes}}
\newcommand{\et}{{\operatorname{\acute{e}t}}}
\newcommand{\proet}{{\operatorname{pro\acute{e}t}}}
\newcommand{\qproet}{{\operatorname{qpro\acute{e}t}}}
\newcommand{\HT}{\operatorname{HT}}

\newcommand{\red}{\mathrm{red}}
\newcommand{\wt}{\widetilde}
\newcommand{\wtOm}{\wt\Omega}
\newcommand{\rk}{\mathrm{rk}}
\renewcommand{\wh}{\widehat}
\renewcommand{\lim}{\varprojlim}
\newcommand{\tf}{[\tfrac{1}{p}]}
\newcommand{\aeq}{\stackrel{a}{=}}\newcommand{\cH}{{\ifmmode \check{H}\else{\v{C}ech}\fi}}

\begin{document}
	\title[The relative Hodge--Tate spectral sequence for rigid spaces]{The relative Hodge--Tate spectral sequence\\for rigid analytic spaces}
	\author{Ben Heuer}
	\maketitle
	
	\section{Introduction}
	\begin{abstract}
		We construct a relative Hodge--Tate spectral sequence for any smooth proper morphism of rigid analytic spaces over a perfectoid field extension of $\Q_p$.  To this end, we generalise Scholze's strategy in the absolute case by using smoothoid adic spaces.  As our main additional ingredient, we prove a perfectoid version of Grothendieck's ``cohomology and base-change''. We also use this to prove local constancy of Hodge numbers in the rigid analytic setting, and deduce that the relative Hodge--Tate spectral sequence degenerates. 
	\end{abstract}
	\subsection{The relative Hodge--Tate sequence}
	Let $C$ be a complete algebraically closed field extension of $\Q_p$ and let $X\to \Spa(C)$ be a smooth proper rigid space over $C$. Then there is a ``Hodge--Tate spectral sequence'', namely a natural $E_2$-spectral sequence of $C$-vector spaces
	\begin{equation}\label{eq:HTss-absolute}
		E_2^{ij}= H^i(X,\Omega^j_X(-j))\Rightarrow H^{i+j}_{\et}(X,\Q_p)\otimes_{\Q_p}C,
	\end{equation}
	where $(-j)$ denotes the Tate twist.
	The existence of this spectral sequence was first conjectured by Tate \cite[\S4.1]{tate1967p}. When $X$ is algebraic and admits a model over a discretely valued subfield $K_0$ with perfect residue field, it was proved by Faltings \cite[III Theorem~4.1]{faltings1988p}, and different proofs were later given by Tsuji \cite{TsujiHT} and Nizio{\l}  \cite{Niziol-semistable}. Scholze then constructed the spectral sequence \Cref{eq:HTss-absolute} in the more general rigid analytic setting \cite[Theorem~3.20]{ScholzeSurvey}. His approach is based on the pro-\'etale site $X_\proet$, which he introduced in \cite[\S3]{Scholze_p-adicHodgeForRigid}.
	
	The Hodge--Tate spectral sequence degenerates at the $E_2$-page: In the arithmetic setting of Faltings, this can be seen by an argument of Tate using Galois cohomology \cite[\S4.1]{tate1967p}. In the above generality, it is a recent result of Bhatt--Morrow--Scholze \cite[Theorem~1.7.(ii)]{BMS}.
	
	\medskip
	
	The aim of this article is to construct a relative version of the Hodge--Tate spectral sequence:
	
	\begin{Theorem}[\Cref{t:relHTSS}]\label{t:intro-HT-spectral-seq}
		Let $K$ be a perfectoid field over $\Q_p$.
		Let $f: X\to  S$ be a smooth proper morphism of reduced rigid spaces over $K$.  Let $ S_{\proet}$ be the pro-\'etale site  with its completed structure sheaf $\wh\O_S$. Let $\nu: S_{\proet}\to  S_{\an}$ be the natural morphism. 
		\begin{enumerate}
			\item  There is a natural first quadrant spectral sequence of $\wh\O_S$-modules on $ S_{\proet}$
			\begin{equation}\label{eq:main-thm-intro}
			 E_2^{ij}=\nu^{\ast}R^if_{\ast}\Omega^j_{ X| S}\{-j\}\Rightarrow (R^{i+j}f_{\proet\ast}\wh\Z_p)\otimes_{\wh\Z_p} \wh\O_S.
			 \end{equation}
			 Here $\wh\Z_p:=\varprojlim_n \nu_X^{\ast}(\Z/p^n)$ for the natural morphism of sites $\nu_X:X_\proet\to X_\an$.
			\item The sheaves $R^if_{\ast}\Omega^j_{ X| S}\{-j\}$ are analytic vector bundles. In contrast, the abutment  $(R^{i+j}f_{\proet\ast}\wh \Z_p)\otimes_{\wh\Z_p} \wh\O$ is instead a pro-\'etale vector bundle.
			\item The spectral sequence in part 1 degenerates at the $E_2$-page.
		\end{enumerate}
	\end{Theorem}
	Here $\{-j\}$ denotes a Breuil--Kisin--Fargues twist over $K$, see \Cref{d:BKF-twist}. This can be canonically identified with a Tate twist $(-j)$ as in \eqref{eq:HTss-absolute} if $K$ contains all $p$-power unit roots.

	For any $n\in \N$, \Cref{t:intro-HT-spectral-seq} equips $(R^{n}f_{\proet\ast}\wh\Z_p)\otimes_{\wh\Z_p}\wh\O$ with a natural Hodge--Tate filtration whose graded pieces are given by the Hodge cohomology groups  $\nu^\ast R^if_{\ast}\Omega^j_{ X| S}\{-j\}$.

	\Cref{t:intro-HT-spectral-seq} is closely related to the following rigid analytic version of a well-known result of Deligne \cite[Th\'eor\`eme~5.5]{Deligne-HTdeg}, the local constancy of Hodge numbers in families:
	
	\begin{Theorem}[\Cref{p:Deligne-constant-fibre-dim}]\label{t:deligne-intro}
		The sheaf $R^if_{\ast}\Omega^j_{ X| S}$ is finite locally free for any $i,j\geq 0$, and its formation commutes with base-change. For varying geometric points $s:\Spa(C,C^+)\to  S$, the Hodge numbers $\dim_CH^i(X_s,\Omega^j_{X_s})$	of the fibres $X_s$
		are therefore locally constant on $ S$.
	\end{Theorem}

	\Cref{t:deligne-intro} implies that the relative Hodge--Tate sequence is preserved by pullback.  In particular, for any geometric point  $s:\Spa(C,C^+)\to S$, the pullback of the spectral sequence \Cref{eq:main-thm-intro} along $s$ is canonically isomorphic to the absolute Hodge--Tate sequence \eqref{eq:HTss-absolute} of the fibre $f_s: X_s\to \Spa(C,C^+)$ of $f$. 
	This means that we can think of \Cref{t:intro-HT-spectral-seq} as providing ``variations of Hodge--Tate structures'' in smooth proper families of rigid spaces.

	In fact, we will show that \Cref{t:intro-HT-spectral-seq} holds without the assumption that $S$ is reduced whenever one knows a priori that  $R^if_{\ast}\Omega^j_{ X| S}$ is locally free, for example in algebraic situations.
	
	\subsection{Previous results on the relative Hodge--Tate spectral sequence}
	Further to the absolute case \Cref{eq:HTss-absolute} discussed in the beginning, which is the case of $S=\Spa(C)$, the following instances of the relative Hodge--Tate  sequence were previously known:
	\begin{enumerate}
		\item Caraiani--Scholze show \Cref{t:intro-HT-spectral-seq}  in the case when $f$ is the base-change to $K$ of a smooth proper morphism of smooth rigid spaces $f_0: X_0\to  S_0$ over a discretely valued subfield $K_0\subseteq K$ with perfect residue field  \cite[\S2.2]{CaraianiScholze}. These conditions appear because the result is deduced from Scholze's relative de Rham comparison isomorphism. 
		\item In the algebraic case that $f$ is the analytification of a smooth projective morphism $f_0:X_0\to S_0$ of varieties over the $p$-adic subfield $K_0$, the relative Hodge--Tate sequence has recently been constructed by Abbes--Gros \cite[Théorème 6.7.5]{AbbesGros_relativeHodgeTate}, by developing systematically a strategy sketched by Faltings \cite{Faltingsalmostetale} based on what they call the relative Faltings topos. There is a local version of this result due to He \cite[\S12]{he2022cohomological}. 
		\item The earliest work on relative Hodge--Tate structures is due to Hyodo: In \cite{Hyodo-HTimperfect}, he constructs the Hodge--Tate sequence for the Tate module of an abelian variety over a discretely valued extension of $\Q_p$ with imperfect residue field. This is related to the case that $f_0$  in 2.\ is a relative abelian variety.  Let us also mention \cite[(3.6)]{HyodoVariations} about variations of Hodge--Tate structures in algebraic smooth proper families over $K_0$.
		\item Koshikawa--Gaisin \cite[\S7.1]{Gaisin-Koshikawa} have recently constructed the relative Hodge--Tate spectral sequence in the case that $f:X\to S$ arises as the generic fibre of a smooth morphism $\mathfrak X\to \mathfrak S$ of admissible formal schemes over $\O_K$ where $K$ is algebraically closed. These additional assumptions appear because their proof is based on a relative version of the $A_{\inf}$-cohomology from Bhatt--Morrow--Scholze's integral $p$-adic Hodge theory \cite{BMS}. This approach is conceptually related to that of \cite{AbbesGros_relativeHodgeTate}.
	\end{enumerate}
	
	We note that much of the work on the relative Hodge--Tate sequence is fairly recent:
	One reason for the time that passed between  \Cref{eq:HTss-absolute} to \Cref{eq:main-thm-intro} is that already the  formulation of the latter requires more advanced structures like the pro-\'etale site or the relative Faltings topos.
	\subsection{Degeneration}
	In the arithmetic cases treated by Caraiani--Scholze and Abbes--Gros, the degeneration of the Hodge--Tate sequence, \Cref{t:intro-HT-spectral-seq}.3 can be shown using the Galois action, generalising Tate's argument in the absolute case. 
	In our more general setting, a different argument is needed:  Our strategy is to first prove
	part~2 of \Cref{t:intro-HT-spectral-seq}. We deduce this from three main ingredients. The first is pro-\'etale-local constancy of  $R^nf_{\proet\ast}\wh{\Z}_p\tf$   which follows from \cite[Theorem~10.5.1]{ScholzeBerkeleyLectureNotes}, the second is Scholze's Primitive Comparison Theorem \cite[Theorem~1.3]{Scholze_p-adicHodgeForRigid}. 
	The third ingredient is a rigid analytic version of Grauert's Theorem which we show in \Cref{t:cohomology-and-base-change-rigid}. From these three, we first prove \Cref{t:deligne-intro}, and then use this to construct the relative Hodge--Tate spectral sequence. Using the base-change property, the degeneration then follows from the degeneration of \Cref{eq:HTss-absolute} due to Bhatt--Morrow--Scholze.

	\subsection{Splittings}
	Given its degeneration, it is a natural question whether the relative Hodge--Tate sequence is split, like in the absolute case.
	It was already observed by Hyodo that this is not the case  \cite[Theorem~3]{Hyodo-HTimperfect}. A beautiful explanation of this phenomenon can be found in the works of \mbox{Liu--Zhu} \cite[\S2]{LiuZhu_RiemannHilbert} and Abbes--Gros \cite[Corollary 5.6]{AbbesGrosSimpsonII}. Both show  (in their respective setups) that one instead obtains a splitting of the underlying vector bundle of the Higgs bundle associated to $(R^{n}f_{\proet\ast}\wh\Z_p)\otimes_{\wh\Z_p} \wh\O$ under the $p$-adic Simpson correspondence. To illustrate this phenomenon, and to investigate the role that choices of a lift will play in the general setting, we  discuss in \S\ref{s:splittings} relative  Hodge--Tate splittings in the case of relative curves.
	
	\subsection{Cohomology and base-change}
	In broad strokes, our proof of the first part of \Cref{t:intro-HT-spectral-seq} follows Scholze's strategy in the absolute case in \cite[\S3]{ScholzeSurvey}, supplemented by two main ingredients: Firstly, we use local results on ``smoothoid adic spaces''. These are perfectoid families of smooth rigid spaces introduced \cite[\S2]{heuer-sheafified-paCS} for applications in relative $p$-adic Hodge theory like the present article. The second main technical ingredient, which we develop in this article, is a  base-change result for proper morphisms to perfectoid spaces:
	
	Let $K$ be any non-archimedean field of residue characteristic $p$.  Let $C$ be a completed algebraic closure of $K$.
	Let $f:X\to S$ be a proper morphism of rigid spaces over $K$. Let $g:S'\to S$ be a morphism of adic spaces and let 
 $X':=X\times_SS'$. The main technical work of this article is about base-change results for the resulting Cartesian diagram:
	\begin{equation}\label{eq:bc-diag}
		\begin{tikzcd}
			X'\arrow[d,"g'"] \arrow[r,"f'"] & S' \arrow[d,"g"] \\
			X\arrow[r,"f"] &  S
		\end{tikzcd}
	\end{equation} We note that it is not clear in general whether this fibre product exists in the category of adic spaces, as sheafiness  can be an issue. But if $f$ is smooth and $S'$ is perfectoid, then $X'$ is in fact an adic space, namely an instance of the aforementioned smoothoid spaces. In this setting,  we prove the following version of Grothendieck's ``cohomology and base-change'':
	
	\begin{Theorem}\label{t:cohomology-and-base-change-perfectoid-intro}
		Let $F$ be an $S$-flat coherent $\O_X$-module.
		Assume that $S'$ is a rigid space, or that $f$ is smooth and $S'$ is sousperfectoid and $F$ is a vector bundle. Assume that $S$ is reduced and that  for some $i\geq 0$,
		\[S(C)\to \mathbb Z,\quad s\mapsto \dim_{C}H^i( X_s,  F_s)\]
		is locally constant.
		Then the following natural base-change map is an isomorphism:
		\[ g^{\ast}R^if_{\ast} F\to R^if'_{\ast}g'^{\ast} F.\]
	\end{Theorem}
	In fact, we also prove some other cases of ``cohomology and base-change'', see \Cref{t:cohomology-and-base-change-perfectoid}.
	
	We are also interested in  \Cref{t:cohomology-and-base-change-perfectoid-intro} for independent applications to non-abelian Hodge theory: In \cite{HX}, we study the $p$-adic Simpson correspondence for curves in terms of moduli v-stacks. By ``abelianization'', it is possible in this context to reduce certain aspects of non-abelian Hodge theory of a smooth proper curve $X$ to the relative abelian Hodge theory of its spectral cover $f:\mathcal X\to \mathcal A$, a non-smooth relative curve over the Hitchin base $\mathcal A$.
	
	The spectral cover is smooth over a dense open subspace of $\mathcal A$. 
	Understanding how a splitting of the Hodge--Tate sequence of $X$ induces a splitting of the relative Hodge--Tate sequence of the spectral curve $f$ over the regular locus is our main motivation for \S\ref{s:splittings}.

	\subsection*{Acknowledgements}
	 Parts of this article were motivated by preparations for \cite{HX}, and we thank Daxin Xu for related discussions. We moreover thank Peter Scholze and Mingjia Zhang for  helpful conversations.
	This project was funded by the Deutsche Forschungsgemeinschaft (DFG, German Research Foundation) -- Project-ID 444845124 -- TRR 326.
	\section{Setup and Recollections}
	Let $K$ be a non-archimedean field of residue characteristic $p> 0$. Let $K^+\subseteq K$ be a ring of integral elements. By a rigid space over $(K,K^+)$ we mean an adic space locally of topologically finite type over $\Spa(K,K^+)$. We call this a rigid space over $K$ when $K^+$ is clear from the context. By \cite[(1.1.11)]{huber2013etale},  there is an equivalence of categories $r$ between quasi-separated rigid spaces in the sense of Tate to quasi-separated rigid spaces over $(K,\O_K)$ in the above sense. It is known that a morphism $f:X\to S$ of rigid spaces is proper in the sense of Kiehl if and only if $r(f)$ is proper \cite[Theorem~3.3.12]{Lutkebohmert_RigidCurves}: This follows from results of Temkin \cite[Corollaries~4.4-4.5]{Temkin_local-properties}, see also \cite[p.5]{HWZ} for some more details on this fact.
	
	\medskip
	
	From $\S4$ on, we will assume that $K$ is a perfectoid field over $\Q_p$.
	For any analytic adic space $X$ over $K$, Scholze constructs in \cite[\S15]{Sch18} an associated diamond $X^\dmd$ over $\Spa(K)$. We can identify this with the sheaf on the category of perfectoid spaces $Y$ over $K$ that sends $Y$ to the morphisms of adic spaces $Y\to X$.
	The diamond $X^\diamondsuit$ is locally spatial, in particular it has a locally spectral topological space $|X^\diamondsuit|$ associated to it, which can be canonically identified with the topological space underlying $X$. We shall also denote this by $X_\an$ when we consider it as a site. We denote the sheaf cohomology of this site by $H^n_\an$ when we want to emphasize the topology, but often drop the $-_\an$ from notation when it is clear from context.
	
	Second, we have an \'etale site $X^\diamondsuit_\et$ which is canonically identified with $X_\et$ by \cite[Lemma~15.6]{Sch18}. For these reasons, we will freely identify $X$ with $X^\diamondsuit$ when dealing with analytic or \'etale cohomology. We denote by $X_\qproet$ and $X_v$ the quasi-pro-\'etale site, respectively the v-site of $X^\dmd$ \cite[\S14]{Sch18}. When $X$ is locally Noetherian, we denote by $X_\proet$ Scholze's flattened pro-\'etale site from \cite[\S3]{Scholze_p-adicHodgeForRigid}. Each of these are ringed sites in a natural way: Here we endow $X_\proet$ with the completed structure sheaf $\wh \O$, while $X_\qproet$ and $X_v$ carry the structure sheaf given by sending perfectoid spaces $Y\to X$ to $\O(Y)$.  We will denote by 
	\[ \nu:X_\proet \to X_\an,\quad \omega:X_\proet \to X_\et,\quad \eta:X_v\to X_\et,\quad \lambda:X_v\to X_\an,\quad \mu:Y_v\to Y_\proet\]
	the natural morphisms of ringed sites. We will use the following definitions from \cite[\S2]{heuer-sheafified-paCS}:
	
	\begin{Definition}
		We call a morphism $f:X\to Y$ of affinoid analytic adic spaces over $K$ standard-\'etale if it is the composition of rational open immersions and finite \'etale maps.
	\end{Definition}
	\begin{Definition}
		We call an adic space $X$ over $K$ smoothoid if it has a cover by open subspaces $U\subseteq X$ that admit a toric chart: This is a standard-\'etale morphism
		$U\to \mathbb T^d\times T$
		where $\mathbb T^d$ is the rigid torus of some dimension $d$ over $K$ and $T$ is an affinoid perfectoid space over $K$.
	\end{Definition}
	\begin{Example}\label{ex:smoothoid-base-change}
		Let $f:X_0\to S_0$ be a smooth morphism of rigid spaces and let $g:S\to S_0$ be a morphism from a perfectoid space. Then $X:=X_0\times_{S_0}S$ is a smoothoid space.
	\end{Example}
	
	We now recall some facts about smoothoid spaces over $K$ and refer to \cite[\S2]{heuer-sheafified-paCS} for details: Any smoothoid space $X$ is sousperfectoid in the sense of \cite{HK}\cite[\S6.3]{ScholzeBerkeleyLectureNotes}. In particular, smoothoid spaces are sheafy. For the associated diamond $X^\diamondsuit$, the identification $X^\diamondsuit_\et=X_\et$ also identifies the structure sheaves $\O_{X^\diamondsuit_\et}=\lambda_{\ast}\O_{X_v}=\O_{X_\et}$, hence it will be harmless to switch back and forth between considering $X$ as an adic space or as a diamond.
	\begin{Definition}[{\cite[Example~4.2]{BMS}}]\label{d:BKF-twist}
		For $j\in \Z$, we define the Breuil--Kisin--Fargues twist \[K\{j\}:=(\ker \theta)^j/(\ker\theta)^{j+1}\tf\] where $\theta:A_{\inf}(K)\to \O_K$ is Fontaine's map. Then $K\{j\}$ is a $K$-vector space of dimension one. If $K$ contains all $p$-power roots of unity, there is a canonical identification $\textstyle K\{1\}=K(1):=K\otimes_{\Z_p}\varprojlim_{n\in\N} \mu_{p^n}(K).$
		For any $K$-vector space $V$, we set $V\{j\}:=V\otimes_KK\{j\}$.
	\end{Definition}
	We note that one can always choose a generator of $\ker \theta$ to trivialise these twists, but it is more natural to keep them, e.g.\ to keep track of Galois actions in arithmetic situations.

	\begin{Proposition}[{\cite[Proposition~2.9]{heuer-sheafified-paCS}}]\label{p:prop28}
		Let $X$ be a smoothoid adic space and recall that we denote by $\lambda: X_v\to X_\an$, $\nu:X_\proet\to X_\an$ the natural morphisms of sites. Let $n\in \N$. Then $R^n\lambda_{\ast}\O=R^n\nu_{\ast}\wh\O$ is a vector bundle on $X_\an$. In the setting of \Cref{ex:smoothoid-base-change}, there is a natural isomorphism
		\[R^n\lambda_{\ast}\O=g^\ast\Omega^n_{X_0|S_0}\{-n\}.\]
	\end{Proposition}
	\begin{Definition}\label{d:def-wtOm}
		We denote the vector bundle $R^n\lambda_{\ast}\O$ on $X_\an$ from \Cref{p:prop28} by $\wtOm^n_{X}$.
	\end{Definition}

	\section{Cohomology and base change in rigid geometry}\label{s:cohom-BC}
	Throughout this section, we allow $K$ to be any non-archimedean field over $\Z_p$. Let $\varpi\in \O_K$ be a pseudo-uniformiser. The aim of this section is to prove a non-archimedean analogue of Grothendieck's well-known ``cohomology and base change'' in rigid and perfectoid geometry over $K$. More precisely, we prove analogues of three closely related algebraic statements, all of which are sometimes referred to as ``cohomology and base-change''.

	\subsection{Base change in Kiehl's Theorem}
	Our first, most basic version of ``cohomology and base-change'' is a non-archimedean analogue  of the following result due to Grothendieck:
	\begin{Proposition}[{\cite[07VJ]{StacksProject}}]\label{p:base-change-proper-map-classical}
		Let $f:X\to S$ be a proper morphism of Noetherian schemes with $S=\Spec(A)$ and let $F$ be an $S$-flat coherent $\O_X$-module. Then there is a bounded complex of finite projective $A$-modules $P^\bullet$ such that the following holds: For any morphism $g:S'=\Spec(B)\to \Spec(A)$, 	let $g':X\times_SS'\to X$ be the base-change of $g$, then
		\[ R\Gamma(X\times_{S}S',g'^\ast F)=P^\bullet\otimes_AB.\]
		
	\end{Proposition}

	Our non-archimedean analogue is the following result, which may also be described as an extension of Kiehl's Theorem \cite[Theorem~3.3]{KiehlProperMapping} explaining its base-change properties:
	
	\begin{Proposition}\label{t:Kiehl-extended}
		Let $f: X\to  S$ be a proper morphism of rigid spaces over $K$. Let $ F$ be a coherent $\O_{ X}$-module that is $ S$-flat. Then there is a cover of $ S$ by affine open subspaces $U=\Spa(A)$ and for each $U$ a bounded complex of finite projective $A$-modules $P^\bullet$, such that the following holds: Let $g:T=\Spa(B)\to U$ be a morphism of adic spaces, let
		\[ \begin{tikzcd}
			X\times_{{S}}T \arrow[d,"g'"] \arrow[r,"f'"] & T \arrow[d,"g"] \\
			X\times_{{S}}U \arrow[r,"f"] & U,
		\end{tikzcd}
		\]
		be the base change diagram of $f$ along $g$,
		and assume one of the following:
		\begin{enumerate}
			\item $T$ is a rigid space, or
			\item $T$ is a sousperfectoid adic space,  $f$ is smooth and $ F$ is a vector bundle on $ X$.
		\end{enumerate}
		Then $X\times_ST$ is a sheafy adic space and we have a natural quasi-isomorphism
			\[ R\Gamma( X\times_{{S}}T,g'^{\ast} F)=P^\bullet \otimes_AB.\]
		If $f$ is of pure dimension $d$, then $P^\bullet$ can be chosen to be concentrated in $[0,d]$.
	\end{Proposition}
	\begin{Remark}\label{r:setting-3-cohom-bc}
		When $S$ is affine, $F$ is a vector bundle, and $f: X= Z\times  S\to S$ is the projection for a smooth proper rigid space $ Z$ over $K$, then we will see that we can take $U= S$.
	\end{Remark}
	
	Our proof in the rigid case will essentially follow Grothendieck's strategy in the algebraic case, supplemented by the combination of various well-known deep results from the theory of formal models due to Raynaud--Gruson, Bosch, L\"utkebohmert, Temkin and Abbes.
	\begin{Remark}
		\begin{enumerate}
			\item In comparison to the algebraic setting of \Cref{p:base-change-proper-map-classical}, the main difficulty in the non-archimedean setting of \Cref{t:Kiehl-extended}  is that the base-change involves the completed tensor product, which is slightly more difficult to control.
			\item The additional technical assumptions in part 2 are just to guarantee that $X\times_ST$ is still a sheafy adic space, and that $g'^\ast F$ is acyclic on affinoids. More generally, it would be sufficient to assume that $g'^{\ast} F$ is stably pseudocoherent. In fact, we could then even allow $K$ to be an equicharacteristic $0$ non-archimedean field.
			\item Beyond such conditions, it should be possible to prove \Cref{t:Kiehl-extended} in much greater generality in Clausen--Scholze's formalism of ``Analytic Geometry'' \cite{AnalyticGeometry}.
			\item \Cref{p:base-change-proper-map-classical} is often phrased in terms of the complex $Rf'_{\ast}g'^{\ast} F$. The above description is more useful in our setup because it avoids any results about acyclicity of coherent sheaves on affinoid perfectoid spaces to deduce from this the individual cohomology groups. But in the other direction, the following is immediate upon sheafification:
		\end{enumerate}
	\end{Remark}
	\begin{Corollary}\label{c:vanishing-in-degree>dim}
		In the situation of \Cref{t:Kiehl-extended}, there is a natural quasi-isomorphism \[Rf'_{\ast}g'^{\ast} F=g^{\ast}\wt{P^\bullet}\] where $\wt P^\bullet$ denotes the perfect complex of $\O_U$-modules associated to the complex of $A$-modules $P^\bullet$. In particular, $Rf'_{\ast}g'^{\ast} F$ is a perfect complex of $\O_{T}$-modules concentrated in degree $[0,d]$.
	\end{Corollary}
	\begin{proof}
		The statement is local on $S$ and thus follows immediately from \Cref{t:Kiehl-extended}.
	\end{proof}
	\begin{proof}[Proof of \cref{t:Kiehl-extended}]
		That $X\times_ST$ is sheafy is clear when $T$ is rigid. In case 2, it follows from \cite[Propositions 6.3.3 and 6.3.4]{ScholzeBerkeleyLectureNotes} that $X\times_ST$ is again sousperfectoid, hence sheafy.
		
		To construct $P^\bullet$, we first reduce to the case that $K^+=\O_K$: Consider the base-change of the diagram along $\Spa(K,\O_K)\hookrightarrow \Spa(K,K^+)$. For any affinoid rigid or smoothoid adic space over $\Spa(K,K^+)$, the pullback to $\Spa(K,\O_K)$ has the same global sections $\O$. In particular, this base change does not change sections of coherent sheaves over affinoid spaces. Second, as the statement is local, we can assume that $S$ is affinoid, so that $X$ is quasi-compact. We may thus assume without loss of generality that $f$ is a morphism of classical rigid spaces.
		
		In this situation, we begin by following the strategy of \cite[Theorem~2.7]{Lutkebohmer-FARAG} to prove Kiehl's Theorem (see also the remarks after \cite[Theorem 3.3.12]{Lutkebohmert_RigidCurves}):

		By Raynaud's Theorem on the existence of formal models \cite[ Corollary 5.9]{BLR-II}, $f$ admits a formal model $\mathfrak f:\mathfrak X\to \mathfrak S$. By \cite[Corollaries~4.4 and 4.5]{Temkin_local-properties}, this is automatically proper. By \cite[Theorem~2.7]{Lutkebohmer-FARAG}\cite[Proposition~4.8.18.(ii)]{AbbesEGR}, the $\O_X$-module $ F$ admits a coherent formal model $\mathfrak F$ on $\mathfrak X$. Due to the flattening results of Raynaud--Gruson, in the version for formal schemes of topologically finite presentation over $\O_K$ due to Abbes \cite[Th\'eor\`eme~5.8.1]{AbbesEGR}, we can after an admissible blow-up assume that $\mathfrak F$ is flat over $\mathfrak S$.
		
		As the statement is local, we can localise on $\mathfrak S$ and assume that $\mathfrak S=\Spf(R)$ is affine. In the setting of \Cref{r:setting-3-cohom-bc}, we can in the beginning choose $\mathfrak S$ affine and let $\mathfrak X:=\mathfrak S\times \mathfrak Z$ where $\mathfrak Z$ is a formal model of $ Z$. As $\mathfrak X\to \mathfrak S$ is then flat, it suffices to flatten $\mathfrak F$ over $\mathfrak X$ so that we can ensure that $\mathfrak S$ remains affine.
		Localising on $T$, we can then assume that $T$ is affinoid rigid, or affinoid perfectoid, so $g$ has an affine formal model $\Spf(S)\to \Spf(R)$. We then set $U:=\mathfrak S_{\eta}$.
		
		Let now $\mathfrak U$ be an affine open cover of $\mathfrak X$. If $f$ is smooth, we may after passing to a further admissible blow-up of $\mathfrak S$ assume that the generic fibre of $\mathfrak U\to \mathfrak S$ is standard smooth, i.e.\ a composition of rational localisations, finite \'etale morphisms, and relative tori. In the setting of \Cref{r:setting-3-cohom-bc}, we can arrange this by just blowing up a cover of $\mathfrak Z$, so again $\mathfrak S$ remains affine.
		
		Since $\mathfrak X\to \mathfrak S$ is separated, any intersection of opens in $\mathfrak U$ is again affine.
		Consequently, the \v{C}ech complex $ C^\bullet:=\check{C}^\bullet(\mathfrak U,\mathfrak F)$ computes $R\Gamma(\mathfrak X,\mathfrak F)$. We wish to compute the cohomology of $g'^{\ast} F$ in terms of the base-change of this \v{C}ech complex. For this we first check:
		\begin{Claim}
			$\check{C}^\bullet(\mathfrak U_\eta\times_UT, g'^\ast F)$ computes $R\Gamma( X\times_{ S}T,g'^\ast F)$
		\end{Claim}
		\begin{proof}
			It suffices to see that $g'^{\ast} F$ is acyclic on affinoid subspaces of $ X\times_{ S}T$. In the rigid case, this is immediate from Kiehl's Theorem B \cite{Kiehl-Theorem-A-und-B}. In the sousperfectoid case, this is in general more subtle, but our additional assumptions imply that $g'^{\ast} F$ is still a finite locally free module. In this case, by \cite[Theorem 1.3.4]{Kedlaya-AWS}, one has the desired acyclicity.
		\end{proof}	
		
		\begin{Claim}\label{cl:check-computes-base-change}
			$\check{C}^\bullet(\mathfrak U_\eta\times_UT, g'^\ast F)=C^\bullet\hat{\otimes}_RS[\tfrac{1}{\varpi}]$.
		\end{Claim}
		\begin{proof}
			Unravelling the definition,  it suffices to prove the following:
			If $\mathfrak V=\Spf(R')\in \mathfrak U$ is any finite intersection of affine opens in the cover, and we set $V:=\mathfrak V_\eta$, then
			\[\O(g'^{-1}(V))= R'\hat{\otimes}_RS[\tfrac{1}{\varpi}],\]
			where $\hat{\otimes}$ is the $\varpi$-adically completed tensor product.
			If $T=\Spa(S[\tfrac{1}{\varpi}])$ is rigid, this is clear. 
			
			When $T$ is sousperfectoid, we use our assumption that the generic fibre $V\to S$ is standard smooth: This ensures that $V\times_UT$ is sheafy, and it allows us to reduce to  the case that $V\to S$ is either of the following: 1) a relative torus, 2) finite \'etale, or 3) a rational open subspace. For 1) and 2) the desired statement is clear. It remains to consider the case that $V\to S$ is an open immersion. In this case, the description follows directly from the definition of the structure presheaf on rational open subspaces.
		\end{proof}

		We are thus left to compute $C^\bullet\hat{\otimes}_RS$. 
		By the Proper Mapping Theorem for $f:\mathfrak X\to \mathfrak S$, \cite[Théorème 2.11.5]{AbbesEGR},  $H^i(C^\bullet)$ is a finitely presented $R$-module for any $i\in \N$. We can therefore use the following lemma: 
		
		\begin{Lemma}[{\cite[\S II.5 Lemma 1]{MumfordAV}}]\label{l:perfect-complex}
			Let $A$ be a commutative ring. We assume that $A$ is Noetherian, or that $A$ is a flat $\O_K$-algebra of topologically finite presentation.
			\begin{enumerate}
				\item 	Let $f:M\to N$ be a morphism of finite flat $A$-modules. Then $\ker f$ is an $A$-module of finite presentation.
				\item 	Let $C^\bullet$ be a bounded complex of flat $A$-modules. If each $H^i(C^\bullet)$ is finitely presented as an $A$-module, then there is a perfect complex $P^\bullet$ with a quasi-isomorphism $P^\bullet\to C^\bullet$.
			\end{enumerate}
		\end{Lemma}
		\begin{proof}
			Part 1 is clear in the Noetherian case. If $A$ is of topologically finite presentation over $\O_K$, it follows from \cite[\S7.3 Lemma 7]{Bosch-lectures}: Since $N$ is $\varpi$-torsionfree, $\ker f\subseteq M$ has $\varpi$-torsionfree cokernel, which already implies that it is a finite $A$-module. By \cite[\S7.3 Theorem 4]{Bosch-lectures}, $M$ is then automatically of finite presentation as it is $\varpi$-torsionfree.
			
			Due to part 1, part 2 can now be seen exactly as in the Noetherian case, cf \cite[\S II.5 Lemma 1]{MumfordAV}, replacing finitely generated modules with finitely presented ones everywhere.
		\end{proof}
		We now use that $\mathfrak F$ is $\mathfrak S$-flat. By \cite[Proposition 1.12.6]{AbbesEGR}, this implies that  $C^\bullet$ is a complex of flat $R$-modules.
		We can thus apply \cref{l:perfect-complex} to find a bounded complex $P^\bullet$ of finite projective $R$-modules with a quasi-isomorphism $h:P^\bullet \to C^\bullet$. It now remains to see:
		\begin{Claim}
			$h\hat{\otimes}_RS:P^\bullet\hat{\otimes}_RS\to C^\bullet \hat{\otimes}_RS$
			is still a quasi-isomorphism.
		\end{Claim} 
		\begin{proof}
			The mapping cone $L^\bullet$ of $h$ is an exact complex that fits into a short exact sequence 
			\[ 0\to C^\bullet\to L^\bullet\to P^\bullet[1] \to 0\]
			of complexes of $A$-modules. As it is split exact in each degree, it remains exact after applying $-\hotimes_R S$. Hence	$L^\bullet\hotimes_R S$ is the mapping cone of $h\hat{\otimes}_RS$. It therefore suffices to see that 
			$L^\bullet\hotimes_R S$
			is still exact. To see this, we use that both $C^\bullet$ and $P^\bullet$ consist of flat $R$-modules, and thus in particular the same is true for $L^\bullet$. The statement now follows from \Cref{p:commute-hotimes-and-H}.2.
			\end{proof}
		This completes the proof  of \cref{t:Kiehl-extended}.
	\end{proof}
	Before going on, we note the following useful application:
	\begin{Corollary}\label{c:cohom-F-product-smooth-proper-perfectoid}
		Let $X$ be a smooth proper rigid space over $K$ and let $Y$ be an affinoid sousperfectoid space over $K$. Let $F$ be an analytic vector bundle on $X$. Then for any $i\geq 0$:
		\begin{enumerate}
			\item $H^i_\an(X\times Y,F)=H^i_\an(X,F)\otimes_K \O(Y)$
			\item $H^i_v(X\times Y,F)=H^i_v(X,F)\otimes_K \O(Y)$
		\end{enumerate}
		In particular, for the structure map $f:X\to \Spa(K)$, we have \[R^if_{v\ast}F=H^i_v(X,F)\otimes_K\mathbb G_a.\]
	\end{Corollary}
	This generalizes \cite[Proposition 4.2]{heuer-diamantine-Picard}, which treats the case of $F=\O$.
	
	\begin{proof}
		We apply \Cref{t:Kiehl-extended} to $f$. Using that $\O(Y)$ is flat over $K$, this shows that
		\[H^i_\an(X\times Y,F)=H^i(P^\bullet \otimes_K\O(Y))=H^i(P^\bullet)\otimes_K\O(Y)=H^i_\an(X,F)\otimes_K \O(Y).\]
		To deduce 2, we use that by \Cref{p:prop28} and \Cref{d:def-wtOm},
		the Leray spectral sequence for the morphism $\lambda:(X\times Y)_v\to (X\times Y)_\an$ is given by
		\[ H^i_\an(X\times Y,F\otimes \wtOm^j)\Rightarrow H^{i+j}_v(X\times Y,F)\]
		where on the left we used the projection formula  $R\lambda_\ast F=F\otimes R\lambda_\ast \O.$
		By part 1,
		\[H^i_\an(X\times Y,F\otimes \wtOm^j)=H^i_\an(X,F\otimes \wtOm^j)\otimes_K\O(Y).\]
		The result follows by comparing to the Leray sequence for $\lambda:X_v\to X_\an$ tensored with $\O(Y)$: The natural map between $E_2$-pages is an isomorphism, hence also between the abutments.
	\end{proof}
	
	\subsection{Grauert's Theorem in rigid analytic geometry}
	Our second version of ``cohomology and base change'' is a rigid analytic version of the  following Theorem of Grothendieck, which is sometimes also referred to as ``Grauert's Theorem'':
	
	\begin{Proposition}[{\cite[\S II.5]{MumfordAV}}]\label{t:cohom-bc-schemes}
		Let $f: X\to  S$ be a proper morphism of locally Noetherian schemes where $S$ is reduced. Let $F$ be an $S$-flat coherent $\O_X$-module. 
		For any $s\in S$, let $k(s)$ be the residue field of $s$, denote by $X_s:=X\times_S\Spec(k(s))$ the fibre of $X$ over $s$ and by $F_s$ the base-change of $F$ along $X_s\to X$. Then for any $i\geq 0$, the following are equivalent:
		\begin{enumerate}
			\item The following function on the topological space underlying $S$ is locally constant:
			\[S\to \Z, \quad s\mapsto \dim_{k(s)}H^i( X_s,F_s)\]
			
			\item The $\O_S$-module $R^if_{\ast} F$ is finite locally free, and for any $s\in  S$, the natural map 
			\[ (R^if_{\ast} F)_s\otimes_{\O_{S,s}} {k(s)}\to H^i( X_s,F_s)\]
			is an isomorphism.
		\end{enumerate}
	\end{Proposition}
	
	For our non-archimedean version, we will replace points $s$ of the scheme $S$ with geometric points of rigid spaces: We therefore now fix  $C=\widehat{\overline{K}}$, the completion of an algebraic closure.
	\begin{Proposition}\label{t:cohomology-and-base-change-rigid}
		Let $f: X\to  S$ be a proper morphism of rigid spaces over $K$ where $S$ is reduced. Let $ F$ be an $S$-flat coherent $\O_X$-module. For any $s\in S(C)$, we denote by $X_s:=X\times_S\Spa(C)$ the fibre of $X$ over $s$ and  by $F_s$ the base-change of $F$ along $X_s\to X$. Then for any $i\geq 0$, the following are equivalent:
		\begin{enumerate}
			\item The following function on $S(C)$ is Zariski-locally constant:
			\[b_i: S(C)\to \Z,\quad s\mapsto \dim_{C}H^i( X_s,  F_s)\]
			\item The $\O_S$-module $R^if_{\ast}F$ is finite locally free, and for any $s\in  S(C)$, the natural map
			\[ (R^if_{\ast} F)_s\otimes_{\O_{S,s}} C\to H^i( X_s, F_s)\]
			is an isomorphism.
		\end{enumerate}
	\end{Proposition}
	While we do not know a reference in the literature for this purely rigid-analytic statement, we do not claim any originality. For example, its possibility is remarked in \cite[p.37 l.1]{Conrad_Ampleness}.
	\begin{proof}
		We clearly have 2 $\Rightarrow$ 1. To see the other direction,
		we follow the original proof of \Cref{t:cohom-bc-schemes} as explained in \cite[\S II.5]{MumfordAV}. The only necessary modification is that we work with $\mathrm{MaxSpec}$ instead of $\Spec$, which is fine due to the following Lemma:
		
		\begin{Lemma}\label{l:red-rigid-space-check-on-fibres}
			Let $X$ be a reduced rigid space over $K$ and let $F$ be a coherent $\O_X$-module. If the function $X(C)\to \Z$, $s\mapsto \dim_{C}F_s$ is constant $=r$, then $F$ is locally free of rank $r$.
		\end{Lemma}
		\begin{proof}
			We use that locally, $X=\Spa(A)$ for a Jacobson ring $A$, and $F$ comes from a finitely generated $A$-module. The Lemma now follows from \cite[\S9, Lemma (3.7)]{ArbarelloGeometryofcurves}.
		\end{proof}
		From this, \Cref{t:cohomology-and-base-change-rigid} can be seen exactly as in \cite[p.50]{MumfordAV}. We reproduce the argument for the reader's convenience: We apply \Cref{t:Kiehl-extended} to the Cartesian diagram
		\[
		\begin{tikzcd}
			X_s \arrow[d] \arrow[r] & \Spa(C) \arrow[d] \\
			X \arrow[r]             & S.          
		\end{tikzcd}\]
		As the statement is local, we may assume that $S=\Spa(A)$ is affine and connected, then \Cref{t:Kiehl-extended} says that after further localisation, there is a perfect complex $P^\bullet$ of finite free $A$-modules such that 
		\begin{equation}\label{eq:bc-of-P-to-fibre}
		R\Gamma(X_s,F_s)=P^\bullet\otimes_AC,
		\end{equation}
		where $A\to C$ is the $K$-algebra homomorphism associated to $s$. If $d^i:P^i\to P^{i+1}$ is the transition map, this shows that  
		\begin{equation}\label{eq:dim-estimate-cohom-bc}
			\dim_CH^i(X_s,F_s)=\dim_C(P^i\otimes_AC)-\dim_C\im(d^i\otimes_AC)-\dim_C\im(d^{i-1}\otimes_AC).
		\end{equation}
		Since $P^i$ is finite free, the first term on the right is constant in $s$. The last two terms are lower-semicontinuous in $s$: Indeed, $d^i\otimes_AC$ has rank $<r$ if and only if $\wedge^r(d^i\otimes_AC)=0$, and this happens on the vanishing set of the matrix coefficients of $\wedge^rd^i$. We have thus shown:
		
		\begin{Lemma}\label{l:upper-semicts}	Let $f: X\to  S$ be a proper morphism of rigid spaces over $K$. Let $ F$ be an $S$-flat coherent $\O_X$-module.
			\begin{enumerate}
				\item For any $i\in \N$, the map \[b_i: S(C)\to \Z,\quad s\mapsto \dim_{C}H^i( X_s,  F_s)\] is upper semi-continuous.
				\item The function
				\[\textstyle S(C)\to \Z,\quad s\mapsto \sum\limits_{i\geq 0}(-1)^i\dim_{C}H^i( X_s,  F_s)\]
			 is Zariski-locally constant on $S(C)$.
			\end{enumerate}
		\end{Lemma}
		
		We note that part 2 is also contained in \cite[
		Theorem A.1.6]{Conrad_Ampleness}. 
		
		The key computation for cohomology and base-change is now the following:
		\begin{Lemma}\label{cl:cbc-constant-fibre-dimension}
			Assume that the function $S(C)\to \Z$, $s\mapsto \dim_{C}H^i( X_s,  F_s)$ is constant for some $i\geq 0$. Then $H^i(P^\bullet)$ is projective and for any $A$-module $B$,
			\[H^i(P^\bullet \otimes_AB)=H^i(P^\bullet) \otimes_AB.\]
		\end{Lemma}
		\begin{proof}
			
			The lower-semicontinuity of the last two terms in  \Cref{eq:dim-estimate-cohom-bc} combines with the assumption  to show that the dimension of $\im(d^i\otimes_AC)$ is locally constant. 
			Thus the same is true for $\coker(d^i\otimes_AC)$. Using right-exactness of $-\otimes_AC$, it follows from \Cref{l:red-rigid-space-check-on-fibres} that $\coker d^i$ is projective. Hence
			\[P^{i+1}=\im d^i\oplus \coker d^i\]
			It follows that $\im d^i$ is projective, hence $P^i=\ker d^i\oplus \im d^i$.
			The same argument applied to $\im(d^{i-1}\otimes_AC)$ shows that \[P^{i-1}=\ker d^{i-1}\oplus \im d^{i-1}.\]
			
			Second, one applies the same reasoning to the map $d'^{i-1}:P^{i-1}\to \ker d^i$ induced by $d^{i-1}$: The dimension of $\im(d'^{i-1}\otimes_AC)=\im(d^{i-1}\otimes_AC)$ is constant, hence $\coker d'^{i-1}=H^i(P^\bullet)$ is projective.  We deduce that there is a splitting
			$\ker d^i=\im d^{i-1}\oplus H^i(P^\bullet)$. All in all, this yields a decomposition
			\[P^i= \im d^{i-1}\oplus H^i(P^\bullet)\oplus \im d^i.\]
			This shows that both $d^{i-1}:P^{i-1}\to P^i$ and $d^{i}:P^i\to P^{i+1}$ are compositions of a split projection with an isomorphism and a split inclusion. Hence $H^i(P^\bullet\otimes_AB)=H^i(P^\bullet)\otimes_AB$.
		\end{proof}
		As $R^if_{\ast}F$ is the coherent $\O_X$-module associated to the $A$-module $H^i(P^\bullet)$, this shows that $R^if_{\ast}F$ is locally free. By  \Cref{eq:bc-of-P-to-fibre},
	 \Cref{cl:cbc-constant-fibre-dimension} applied to the map $A\to C$  shows the desired isomorphism in \Cref{t:cohomology-and-base-change-rigid}.2. This finishes the proof of \Cref{t:cohomology-and-base-change-rigid}.
	\end{proof}

	\subsection{Cohomology and base change}
	Putting everything together, we can deduce the non-archimedean version of ``cohomology and base-change'' which is the main result of this section:
	
	\begin{Theorem}\label{t:cohomology-and-base-change-perfectoid}
		Let $K$ be any non-archimedean field over $\Z_p$ and let $C$ be the completion of an algebraic closure of $K$. Let
		\[\begin{tikzcd}
			X'\arrow[d,"g'"] \arrow[r,"f'"] & S' \arrow[d,"g"] \\
			X\arrow[r,"f"] &  S
		\end{tikzcd} \]
		be a Cartesian diagram of adic spaces over $K$
		where $f$ is a proper morphism of rigid spaces. Let $F$ be an $S$-flat coherent $\O_X$-module. Let $i\in \N$. Assume one of the following:
		\begin{enumerate}
			\item $S'$ is a rigid space, or
			\item $S'$ is sousperfectoid, $f$ is smooth and $F$ is locally free.
		\end{enumerate}
		Then the base-change map
		\[ g^{\ast}R^if_{\ast} F\to R^if'_{\ast}g'^{\ast} F\] 
		is an isomorphism in each of the following cases:
		\begin{enumerate}[label=(\alph*)]
			
			\item $S$ is reduced and the following function from  \Cref{t:cohomology-and-base-change-rigid}.1 is locally constant
			\[b_i: S(C)\to \mathbb Z,\quad s\mapsto \dim_{C}H^i( X_s,  F_s),\]
			\item $S'$ is a rigid space and $g$ is flat, or
			\item $S'$ is perfectoid, $g$ is a pro-\'etale morphism in $S_\proet$, $S$ is smooth and $\mathrm{char}(K)=0$.
		\end{enumerate}
	\end{Theorem}
	\begin{Remark}
		\begin{enumerate}
			\item 
	Note that in the case of assumption 1, we do not assume that $g$ is a morphism of rigid spaces, i.e.\ $S'$ could be a rigid space over any complete field extension $K'$ of $K$. In particular, this generalises \cite[
	Theorem~A.1.2]{Conrad_Ampleness} which is the case that $S'\to S$ is given by an extension of base fields $\Spa(K')\to \Spa(K)$.
	\item With more work, one can show that the assumption on $\mathrm{char}(K)$ in (c) is not necessary.
	\end{enumerate}
	\end{Remark}
	
	\begin{proof}
		As the statement is local, we may shrink $S$ and $S'$ so that \Cref{t:Kiehl-extended} applies. Then	part (b) is immediate from \Cref{t:Kiehl-extended}. Part (a) follows from \Cref{t:Kiehl-extended} combined with \Cref{cl:cbc-constant-fibre-dimension}.
		Finally, the case of (c) follows from the following technical lemma.
	\end{proof}
	
	\begin{Lemma}
		Let $S=\Spa(A,A^+)$ be a smooth rigid space that admits a standard-\'etale map to $\mathbb T^d=\Spa(K\langle T_1^{\pm 1},\dots,T_d^{\pm 1}\rangle)$ for some $d\in \N$. Let $P^\bullet$ be a bounded complex of finite free $A$-modules. Let $S'=\Spa(B,B^+)\to S$ be an affinoid perfectoid pro-\'etale cover in $S_\proet$. Then for any $i\in \Z$,
		\[ H^i(P^\bullet \otimes_AB)=H^i(P^\bullet)\otimes_AB.\]
	\end{Lemma}
	\begin{proof} 
		Recall that any pro-\'etale map can always be written as a composition $g_1\circ g_2$ where $g_1$ is \'etale and $g_2$ is pro-finite-\'etale. Since the statement  for $g_1$ follows from \Cref{t:cohomology-and-base-change-perfectoid}.(b), we may thus assume that $g$ is pro-finite-\'etale and $S'$ is affinoid perfectoid. Write $\mathbb T^d=\Spa(R)$ and consider the pro-\'etale cover $\mathbb T^d_\infty=\Spa(R_\infty,R_\infty^+)\to \mathbb T^d$ where $R_\infty:=K\langle T_1^{\pm 1/p^\infty},\dots,T_d^{\pm 1/p^\infty}\rangle$. 
		Let $S_\infty=\Spa(A_\infty,A_\infty^+)\to S$ and $S'_\infty=\Spa(B_\infty,B_\infty^+)\to S'$ be the base-changes along this cover, so we have Cartesian diagrams
		\[
		\begin{tikzcd}
			S'_\infty \arrow[r] \arrow[d] & S_\infty \arrow[d]\arrow[r]  &\mathbb T^d_\infty \arrow[d] \\
			S' \arrow[r]                  & S \arrow[r]                  & \mathbb T^d.                 
		\end{tikzcd}\]
		We now first prove the  result for the cover $S_\infty\to S$: By \cite[Lemma~4.5]{Scholze_p-adicHodgeForRigid}, the kernel and cokernel of the natural map \[A^+\hotimes_{R^+}R^+_\infty\to A_\infty^+\]
		are killed by $\varpi^k$ for some $k$ (see also \cite[Lemma~2.14]{heuer-sheafified-paCS} for the claim about the kernel).
		
		Choose an integral model $P^{+,\bullet}$ of finite free $A^+$-modules such that $P^{+,\bullet}\tf=P^\bullet$. 
		As a consequence of \Cref{l:perfect-complex}.1, for any $j\in \N$, the $A$-module $H^{j}(P^{+,\bullet})$ is finitely presented, in particular it has bounded $p$-torsion.
		We now use that for any $n$, the map $R^+/p^n\to R^+_\infty/p^n$ is almost  flat:  By \Cref{p:commute-hotimes-and-H}.1, it follows from these properties that
		\[
		 H^i(P^\bullet \otimes_AA_\infty)=H^i(P^{+,\bullet} \hotimes_{R^+}R^+_\infty)\tf=H^i(P^{+,\bullet}) \hotimes_{R^+}R^+_\infty\tf=H^i(P^\bullet)\otimes_{A}A_\infty.\]
		We moreover note that this shows that $H^i(P^{+,\bullet} \hotimes_{R^+}R^+_\infty)\aeq H^i(P^{+,\bullet}) \hotimes_{R^+}R^+_\infty$ has bounded $p$-torsion, hence the same is true for $H^i(P^{+,\bullet} \otimes_{A^+}A^+_\infty)$.
		
		As a second step, we now consider the cover $S_\infty'\to S_\infty$. Since this is a pro-finite-\'etale cover of affinoid perfectoids, the map $A_\infty^+/p^n\to B_\infty^+/p^n$ is a colimit of almost finite \'etale maps by almost purity \cite[Theorem 7.9]{perfectoid-spaces}, hence it is almost flat. Again by \Cref{p:commute-hotimes-and-H}.1, we deduce from this and the fact that $H^i(P^{+,\bullet} \otimes_{A^+}A^+_\infty)$ has bounded $p$-torsion that
		\[H^i(P^\bullet\otimes_{A}B_\infty)=H^i(P^\bullet\otimes_{A}A_\infty\otimes_{A_\infty}B_\infty)=H^i(P^\bullet\otimes_{A}A_\infty)\otimes_{A_\infty}B_\infty=H^i(P^\bullet)\otimes_{A}B_\infty.\] 
		Slightly more precisely, by replacing $C^\bullet$ with the truncated complex $d^i:C^i\to C^{i+1}$, this shows that $-\otimes_AB_\infty$ commutes already with the formation of boundaries and cycles of $C^\bullet$.
		
		It remains to descend from $S'_\infty$ to $S'=\Spa(B)$. For this we use that $S_\infty'\to S'$ is an affinoid pro-finite-\'etale $\Z_p^d$-torsor of affinoid perfectoid spaces, so the map $B\to B_\infty$ is module-split. Hence for any $A$-module $M$, the map $M\otimes_AB\to M\otimes_AB_\infty$ is injective. It follows that for any injection $M_1\hookrightarrow M_2$ of $A$-modules, if $	M_1\otimes_AB_\infty \to M_2\otimes_AB_\infty$ is injective, then so is $M_1\otimes_AB \to M_2\otimes_AB$. Consequently, as $-\otimes_AB_\infty$ preserves all short exact sequences computing $H^i(P^\bullet)$ in terms of boundaries and cycles, this  implies that so does $-\otimes_AB$.
	\end{proof}
	\begin{Corollary}
		In the setting of \Cref{t:cohomology-and-base-change-perfectoid}, assume that $S$ is reduced. Then there is for each $i\in \N$ a dense Zariski-open locus $U\subseteq S$ over which $R^if_{\ast}F$ is finite locally free. Moreover, the base-change map in \Cref{t:cohomology-and-base-change-perfectoid} is an isomorphism for any morphism from a rigid space $S'\to S$ that factors through $U$.
	\end{Corollary}
	\begin{proof}
		By \cref{l:upper-semicts}.1, there is a non-empty Zariski-open locus $U$ where $\dim H^i_C(X_s,F_s)$ is constant. By \Cref{t:cohomology-and-base-change-perfectoid}.(a), this has the desired properties.
	\end{proof}
	
	\section{A limit version of the Primitive Comparison Theorem}
	
	Let $C$ be an algebraically closed non-archimedean field extension of $\Q_p$.
	For any smooth proper rigid space $X$ over $C$, Scholze constructed  a spectral sequence  \cite[Theorem 3.20]{ScholzeSurvey}
	\begin{equation}\label{eq:HTss}
		H^i_\an(X,\Omega^j_X(-j))\Rightarrow H^{i+j}_\et(X,\Q_p)\otimes_{\Q_p} C.
	\end{equation}
	As a preparation, we briefly recall Scholzes' strategy, which builds on the earlier work of Faltings \cite{faltings1988p}: The key idea is to realize \eqref{eq:HTss} as a Leray sequence for the morphism of sites
	\[ \nu:X_\proet \to X_\an.\]
	Indeed, Scholze proves that for any $j\geq 0$, there is a canonical isomorphism of $\O_{X}$-modules
	\begin{equation}\label{eq:Rnnu_astO}
		R^j\nu_{\ast} \wh \O_X=\Omega_X^j(-j).
	\end{equation}
	The second main ingredient is the Primitive Comparison   Theorem  \cite[Theorem~5.1]{Scholze_p-adicHodgeForRigid}, which says that the following natural morphism is an isomorphism:
	\begin{equation}\label{eq:PCT}
		H^n_\et(X,\Q_p)\otimes_{\Q_p}K\to H^n_\proet(X,\O)
	\end{equation}
	Together, these two statements combine to give the desired sequence \Cref{eq:HTss}.
	
	Let now $f:X\to S$ be a smooth proper morphism of rigid spaces over a perfectoid field extension of $\Q_p$. The main goal of this article is to construct a relative version of \Cref{eq:HTss} for the morphism $f$.
	As a preparation for this, we have already proved the relative analogue of \Cref{eq:Rnnu_astO}  in \cite{heuer-sheafified-paCS},  which we recalled in \Cref{p:prop28}. The goal of this section is to deduce from further Theorems of Scholze the following relative generalisation of \Cref{eq:PCT}:
	
	Throughout this section, let $K$ be any non-archimedean field over $\Q_p$. For any rigid space $Y$ over $K$, we set $\wh\Z_p:=\varprojlim_{n\in \N} \nu^{\ast}(\Z/p^n)$ where $\nu:Y_\proet\to Y_\an$ is the natural map.
	
	\begin{Theorem}\label{c:PCT-limit} 
		Let $f:X\to S$ be a smooth proper morphism of rigid spaces over $K$.
		\begin{enumerate}
			\item For any $n\in \N$, the following natural map is an isomorphism on $S_\proet$:
			\[ (R^nf_{\proet\ast}\wh\Z_p)\otimes_{\wh\Z_p} \wh\O_S\to R^nf_{\proet\ast}\wh\O_X.\]
			\item The $\wh \O_S$-module $R^nf_{\proet\ast}\wh\O_X$ is finite locally free on $S_\proet$.
			\item For any geometric point $s:\Spa(C,C^+)\to S$ with fibre $X_s:=X\times_Ss$, we have \[s^\ast(R^nf_{\proet\ast}\wh\O_X)=H^n_{\proet}(X_s,\wh\O_{X_s}).\]
		\end{enumerate}
	\end{Theorem}
	For the proof we will need the following result due to Scholze--Weinstein:

	\begin{Theorem}[{\cite[Theorem 10.5.1]{ScholzeBerkeleyLectureNotes}}]\label{t:Sch10.5.1}
		Let $f:X\to S$ be a smooth proper morphism of rigid spaces over $K$. Let $\mathbb L$ be an  $\F_p$-local system on $X_\et$. Then  $R^nf_{\et\ast}\mathbb L$ is an $\F_p$-local system on $S_\et$ for any $n\in \N$.
	\end{Theorem}
	\begin{proof}
		Let $C$ be the completion of an algebraic closure of $K$.
		Scholze--Weinstein prove this Theorem in the case that $K=C$. The general case is easy to deduce: Let $f_C:X_C\to S_C$ be the base-change and let $g:S_C\to S$ be the natural map, then by \cite[Corollary~16.10.(ii)]{Sch18},
		\[g^\ast R^nf_{\et\ast}\mathbb L=R^nf_{C,\et\ast}\mathbb L.\]
		This shows that $R^nf_{\et\ast}\mathbb L$  is pro-\'etale-locally constant. By  \cite[Corollary~16.10.(i)]{Sch18}, this implies that it is \'etale-locally constant. 
	\end{proof}
	\begin{proof}[Proof of \Cref{c:PCT-limit}]
		It follows from \Cref{t:Sch10.5.1} that  $R^nf_{\et\ast}(\Z/p^k)$ is an \'etale-locally constant $\Z/p^k$-module of finite type for any $k\in \N$.
		Consider the morphism $\omega:S_{\proet}\to S_{\et}$. By \cite[Corollary~3.17]{Scholze_p-adicHodgeForRigid},
		\[	R^nf_{\proet\ast}(\Z/p^k)=\omega^{\ast}R^nf_{\et\ast}(\Z/p^k)\]
		is still constant on a finite \'etale cover. Since sequential limits of finite \'etale maps exist in the pro-\'etale site, this shows that there is a pro-\'etale cover $S'\to S$ in $S_\proet$ on which  $R^nf_{\proet\ast}(\Z/p^k)$ becomes constant for all $k\in \N$. After passing to a further cover, we may assume that $H^j(S',\Z/p^k)=0$ for all $j>1$ and $k\in \N$: Indeed, for $j=1$ this can be achieved by forming the limit over all finite \'etale covers of $p$-power degree with a fixed choice of base-point. For $j>1$, it can be achieved by making $S'$ affinoid perfectoid, which ensures that  $H^j(S',\F_p)=0$ by the Artin--Schreier sequence.
		We can therefore invoke \cite[Lemma~3.18]{Scholze_p-adicHodgeForRigid} to see that for any $j\geq 1$,
		\[R^j\plim_{k\in \N} 	R^nf_{\proet\ast}(\Z/p^k) =0.\]
		It follows by a Milnor spectral sequence that 
		$R^nf_{\proet\ast}\wh\Z_p=\plim_{k\in \N} R^nf_{\proet\ast}(\Z/p^k)$
		is pro-\'etale-locally the pro-constant sheaf associated to a finitely generated $ \Z_p$-module. In particular, 
		 $R^nf_{\proet\ast}\wh\Z_p\tf$ is a finite locally free $\wh \Q_p:=\wh\Z_p\tf$-module on $X_\proet$.
		
		We now use the Primitive Comparison Theorem \cite[Theorem~5.1, Corollary~5.11]{Scholze_p-adicHodgeForRigid}, in its pro-\'etale variant \cite[Theorem~4.2]{heuer-PCT-char-p}, which says that for every $k\in \N$,
		\[R^nf_{\proet\ast}(\O_X^{+}/p^k)\aeq R^nf_{\proet\ast}\Z/p^k\otimes \O^+_S/p^k.\]
		 Using again  \cite[Lemma~3.18]{Scholze_p-adicHodgeForRigid}, it follows that also $R^j\varprojlim_{k\in \N}R^nf_{\proet\ast}\O^{+a}_X/p^k$ vanishes for $j\geq 1$: Indeed, on the same basis of spaces $S'$ as above, \[(R^nf_{\proet\ast}\O^{+a}_X/p^k)_{|S'}=(R^nf_{\proet\ast}\Z/p^k)_{|S'}\otimes \O^+/p^k\]  is a constant sheaf tensored with the almost acyclic sheaf $\O^+/p^k$. All in all, we deduce that
		\[
			R^nf_{\proet\ast}\wh\O^+_X\aeq \plim_{k} R^nf_{\proet\ast}(\O^+_X/p^k)\aeq \plim_{k} (R^nf_{\proet\ast}\Z/p^k)\otimes\O^+_S/p^k= R^nf_{\proet\ast}\wh\Z_p\otimes \wh\O_S^+.\]
		After inverting $p$, since $R^nf_{\proet\ast}\wh\Q_p$ is a finite locally free $\wh \Q_p$-module, we see that $R^nf_{\proet\ast}\wh\O$ is a pro-\'etale locally free $\wh\O_S$-module. In other words, it is a vector bundle on $X_\proet$.
		
		Finally, for part 3, we use that for any geometric point $s:\Spa(C,C^+)\to S$, the induced morphism of topoi $s:\Spa(C,C^+)^{\sim}_\proet\to S_\proet^{\sim}$  of \cite[Proposition~3.13]{Scholze_p-adicHodgeForRigid} commutes with limits, which follows from \cite[Lemma~14.4]{Sch18}. By part 1, it therefore suffices to prove that \[s^{-1}(R^nf_{\proet\ast}\Z/p^k)=H^n(X_s,\Z/p^k).\] 
		This follows from \cite[Proposition~2.6.1]{huber2013etale} since $s^{-1}(R^nf_{\proet\ast}\Z/p^k)=(R^nf_{\et\ast}\Z/p^k)_s$. 
	\end{proof}
	By the same argument, we obtain a v-topological version of \Cref{c:PCT-limit}: 
	
	\begin{Corollary}\label{c:PCT-limit-vversion}
		The sheaf $R^nf_{v\ast}\O$ on $S_v$ is a v-vector bundle and we have an isomorphism
		\[ R^nf_{v\ast}\wh \Z_p\otimes_{\wh \Z_p}\O= R^nf_{v\ast}\O.\]
	\end{Corollary}
	\begin{proof}
		Exactly as in \Cref{c:PCT-limit}, replacing \cite[Corollary~3.17]{Scholze_p-adicHodgeForRigid} by \cite[Corollary~5.5]{heuer-PCT-char-p}, and using repleteness of $S_v$ to deduce directly that $R^j\plim_kR^nf_{v\ast}\Z/p^k=0$ for $j>0$.
	\end{proof}
	\begin{Lemma}\label{l:compare-proet-v-Zp}
		Let $f: X\to  S$ be a smooth proper morphism of rigid spaces over $K$. Then
		\[R^nf_{v\ast}\wh \Z_p=\mu^\ast R^nf_{\proet\ast}\wh \Z_p = (\plim_k\eta^\ast R^nf_{\et\ast}\Z/p^k)\otimes_{\wh\Z_p} \O \]
		for any $n\in \N$, where $\mu:S_v\to S_\proet$  and $\eta:X_v\to X_\et$  are the natural maps.
	\end{Lemma}
	\begin{proof}
		For $\Z/p^n\Z$, this holds by \cite[Corollary~5.5]{heuer-PCT-char-p}. The lemma follows in the limit.
	\end{proof}
	\begin{Corollary}\label{c:comp-proet-v-cohom}
		We have 
		$\mu^{\ast}R^nf_{\proet\ast} \wh\O_X=R^nf_{v\ast}\O.$
	\end{Corollary}
	\begin{proof}
		This follows from comparing \Cref{c:PCT-limit}.1 and \Cref{c:PCT-limit-vversion} via \Cref{l:compare-proet-v-Zp}.
	\end{proof}
	
	\section{The relative $p$-adic Hodge--Tate sequence}

	We now put everything together to construct the following relative generalisation of the Hodge--Tate spectral sequence \Cref{eq:HTss}, which is the main result of this article:
	
	\begin{Theorem}\label{t:relHTSS}
	Let $K$ be a perfectoid field extension of $\Q_p$.
	Let $f: X\to  S$ be a smooth proper morphism of reduced rigid spaces over $K$. Let $\nu:X_\proet\to X_\an$ be the natural morphism.
	\begin{enumerate}
		\item  There is a first quadrant spectral sequence of sheaves on $ S_\proet$
		\begin{equation}\label{eq:rel-HT-in-5}
		E_2^{ij}=\nu^{\ast}R^if_{\an\ast}\Omega^j_{ X| S}\{-j\}\Rightarrow (R^{i+j}f_{\proet\ast}\wh\Z_p)\otimes_{\wh\Z_p} \wh\O_S
		\end{equation}
			If $f$ has pure dimension $d$, this sequence is concentrated in degrees $i,j\in [0,d]$.
		\item The $R^if_{\an\ast}\Omega^j_{ X| S}\{-j\}$ are analytic vector bundles, whereas  $(R^{i+j}f_{\proet\ast}\wh \Z_p)\otimes_{\wh\Z_p} \wh\O_S$ is a pro-\'etale vector bundle.
		\item The spectral sequence \Cref{eq:rel-HT-in-5} degenerates at the $E_2$-page.
		\item  The  spectral sequence \Cref{eq:rel-HT-in-5} is natural in $f$. Namely, for any commutative diagram
		\[
		\begin{tikzcd}
			Y \arrow[d] \arrow[r,"g"] &  T \arrow[d,"h"] \\
			X \arrow[r,"f"]            &  S           
		\end{tikzcd}\]
		where $g$ is a smooth proper morphism of rigid spaces over any perfectoid field $K'|K$, there is a canonical morphism $h^\ast E(f)\to E(g)$ between the associated spectral sequences, compatible with  composition. It is an isomorphism if the square is Cartesian.
	\end{enumerate}
	\end{Theorem}
	\begin{Corollary}
		Let $n\in \N$.
	There exists on the pro-\'etale $\wh\O_S$-module $(R^{n}f_{\proet\ast}\wh\Z_p)\otimes_{\wh\Z_p} \wh\O_S$ a natural ``Hodge filtration'' with graded pieces isomorphic to  $\nu^{\ast}R^if_{\an\ast}\Omega^j_{ X| S}\{-j\}$ for $i+j=n$.
	\end{Corollary}
	\begin{Remark}
	\begin{enumerate}
		\item The assumption that $K$ is perfectoid is only necessary to make sense of $\Omega^j_{X|S}\{-j\}$ as an analytic vector bundle on $X$. Alternatively, it suffices that $K$ contains all $p$-power unit roots, then we can replace $\Omega^j_{X|S}\{-j\}$ by the Tate twist $\Omega^j_{X|S}(-j)$.
		\item We note that already the case of $S=\Spa(K)$ yields a generalisation of the absolute Hodge--Tate sequence to more general perfectoid base fields. However, we caution that if $K$ is not algebraically closed, it is in general not true  that $(R^{i+j}f_{\proet\ast}\wh\Z_p)\otimes_{\wh\Z_p} \wh\O$ is the vector bundle associated to $H^{i+j}_{\et}(X,\Z_p)\otimes_{\Z_p} K$ on $\Spa(K)$.
	\end{enumerate}
	\end{Remark}
	\begin{Remark}
	As explained in detail in the introduction, closely related results to part 1 have been obtained by Abbes--Gros \cite{AbbesGros_relativeHodgeTate}, Caraiani--Scholze \cite[Corollary~2.2.4]{CaraianiScholze} and He \cite[Theorem 12.7]{he2022cohomological}. All of these work in the case that $ X\to  S$ is defined over a discretely valued field $K_0$ with perfect residue field, where there is an additional Galois action  to take into account. We are working here in a more general geometric situation, where there is not necessarily any Galois action. However, in our setup, if a model over $K_0$ is given, we recover the Galois action by $\mathrm{Gal}(K|K_0)$  on the sequence using 4., the naturality of the construction.
	
	Second, we also mentioned in the introduction  the relative Hodge--Tate sequence of Gaisin--Koshikawa \cite[Theorem~7.2]{Gaisin-Koshikawa}: This assumes that $K$ is algebraically closed, and that $f$ admits a smooth proper admissible formal model. So in comparison, we can remove these additional assumptions. Furthermore, in \Cref{p:Deligne-constant-fibre-dim}, we show that the additional condition imposed in \cite{Gaisin-Koshikawa} that $R^if_{\an\ast}\wtOm^j_{X|S}$ is locally free is satisfied  when $S$ is reduced.
	\end{Remark}
	\begin{Remark}
	The fact that there exist non-analytic pro-\'etale vector bundles that are extensions of analytic vector bundles is not surprising: Via the local $p$-adic Simpson correspondence (e.g.\ in its version \cite[Theorem~6.5]{heuer-sheafified-paCS}), these locally correspond to nilpotent Higgs bundles.
	\end{Remark}
	
	We recall that based on an earlier construction of Faltings, Scholze's idea for the Hodge--Tate sequence in the absolute case is to realise this sequence as the Leray sequence for the morphism $\nu_X:X_\proet\to X_\an$. For our generalisation to the relative case, the basic idea is to instead consider the commutative diagram of sites
	\[\begin{tikzcd}
	X_\proet\arrow[d,"\nu"]\arrow[r,"f_\proet"]&  S_\proet\arrow[d,"\nu"] \\
	X_{\an} \arrow[r,"f_{\an}"] &    S_{\an}.
	\end{tikzcd}\]

	\begin{proof}[Proof of \Cref{t:relHTSS}]
	In the case that $ S=\Spa(C,\O_C)$, the spectral sequence \Cref{eq:rel-HT-in-5} is due to Scholze \cite[Theorem~3.20]{ScholzeSurvey}, and the degeneration is due to Bhatt--Morrow--Scholze \cite[Theorem 1.7.(ii)]{BMS}. 
	We begin our proof of the relative version by deducing the following $p$-adic analytic analogue of a Theorem of Deligne for schemes \cite[Th\'eor\`eme~5.5]{Deligne-HTdeg}:
	
	\begin{Theorem}\label{p:Deligne-constant-fibre-dim}
		Let $C$ be the completion of an algebraic closure of $K$.
		Let $f: X\to  S$ be a smooth proper morphism of reduced rigid spaces over $K$. Let $i,j\geq 0$.
		\begin{enumerate}
		\item For varying $s\in  S(C)$, the Hodge numbers \[h^{j,i}(X_s):=\dim_CH^i(X_s,\Omega^j_{X_s})\]
		of the geometric fibres of $ X\to  S$ are Zariski-locally constant on $ S$.
		\item 	The $\O_S$-module $R^if_{\ast}\Omega^j_{X|S}$ is finite locally free.
		\item For any morphism $g:S'\to S$ from a rigid or perfectoid space, the base-change map
		\[ g^{\ast}R^if_{\ast}\Omega^j_{X|S}\to R^if'_{\ast}\Omega^j_{X'|S'}\]
		is an isomorphism, where $f':X':=X\times_SS'\to S'$ is the base-change of $f$.
	\end{enumerate}
	\end{Theorem}
	Here in 3, we use that there is a general notion of K\"ahler differentials for morphisms of topologically finite type, which is compatible with base-change. Alternatively, when $S'$ is perfectoid, we could for simplicity  define $\Omega^j_{X'|S'}$ to be the pullback of $\Omega^j_{X|S}$ along $X'\to X$.
	\begin{proof}
		Based on our preparations from \S3, we can take inspiration from Deligne's strategy, replacing the de Rham with the Hodge--Tate sequence:
		By \Cref{l:upper-semicts}.1, the functions
		\[  S(C)\to \mathbb Z, \quad s\mapsto \dim_CH^i(X_s,\Omega^j_{X_s})\]
		are upper semi-continuous for any $i,j\geq 0$. On the other hand, by \Cref{t:relHTSS} in the case of $ S=\Spa(C,\O_C)$ due to Scholze and Bhatt--Morrow--Scholze, for any fixed $n$, we have
		\[ \textstyle\sum_{i+j=n}\dim_CH^i(X_s,\Omega^j_{X_s})=\dim_{C}H^n_{\proet}(X_s,\wh\O_{X_s}).\]
		By \Cref{c:PCT-limit}.2-3, the right hand side is the fibre dimension of the pro-\'etale vector bundle $R^nf_{\proet\ast}\wh\O_X$ on $ S_{\proet}$. It follows that this is Zariski-locally constant on $ S$. In combination, it follows from upper semicontinuity that each of the summands is constant. This proves 1.
		
			Part 2 follows from part 1 by \Cref{t:cohomology-and-base-change-rigid}.
			
			Part 3 follows from part 1 by 
			\Cref{t:cohomology-and-base-change-perfectoid}.(a).
	\end{proof}
	
	Turning to the proof of \Cref{t:relHTSS},
	part 2 is given by \Cref{p:Deligne-constant-fibre-dim}.2 and \Cref{c:PCT-limit}.
	For part 1, let $\wtOm^j_{ X| S}:=\Omega^j_{ X| S}\{-j\}$. We consider for any affinoid perfectoid object $g:S'\to  S$ in $ S_\proet$ with base change $g':X':=X\times_{ S}S'\to  X$ the Leray sequence for the morphism $X'_{\proet}\to X'_{\an}$. By \Cref{p:prop28}, this is of the form
	\begin{equation}\label{eq:leray-in-proof-of-relHTSS}
		H^i_\an( X',g'^\ast\wtOm^j_{ X| S})\Rightarrow H^{i+j}_\proet( X',\wh\O_X).
	\end{equation} The pro-\'etale sheafification in $S'$ of the abutment is $R^{i+j}f_{\proet\ast}\wh\O_X$. By \Cref{c:PCT-limit}, this is
	\[ R^{i+j}f_{\proet\ast}\wh\O_X=R^{i+j}f_{\proet\ast}\wh{\Z}_p\otimes_{\wh{\Z}_p}\wh\O_S.\]
	Hence the sheafification of the abutment of \Cref{eq:leray-in-proof-of-relHTSS} in $S'$ gives the desired abutment of \Cref{eq:rel-HT-in-5}.

	We now turn to the term on the left hand side of  \Cref{eq:leray-in-proof-of-relHTSS}: Its sheafification with respect to $S'_\an$ is 	$R^if'_{\ast}g'^\ast\wtOm^j_{ X| S}$ where $f':X_{S'}\to S'$ is the natural map. By \Cref{p:Deligne-constant-fibre-dim}.3, this equals
	\begin{equation}\label{eq:bc-in-proof-of-rel-HT}
		R^if'_{\ast}g'^\ast\wtOm^j_{ X| S} =g^\ast  R^if_{\ast}\wtOm^j_{ X| S}.
	\end{equation}
	It follows that the pro-\'etale sheafification of this term on $S_\proet$ is equal to $\nu^{\ast}R^if_{\ast}\wtOm^j_{ X| S}$. We thus obtain the spectral sequence described in 1.
	
	We note that in the special case of $ S=\Spa(C)$, this specialises to Scholze's construction.
	
	The naturality described in part 4 follows from the construction, and the base-change property for Cartesian squares follows from \Cref{t:cohomology-and-base-change-perfectoid}.(a).
	  In particular, it follows that for any geometric point $s\in  S(C)$, the fibre of $E$ over $s:\Spa(C)\to  X$ is canonically isomorphic to the Hodge--Tate sequence \eqref{eq:HTss-absolute} of the fibre $f_s: X_s\to \Spa(C)$ of $f$.
	
	To see part 3, we need to show the vanishing of the morphism of vector bundles
	\begin{equation}\label{eq:deg-map-relHTSS}
		\nu^{\ast}R^if_{\ast}\wtOm^j_{ X| S}\to\nu^{\ast}R^{i+2}f_{\ast}\wtOm^{j-1}_{ X| S}
	\end{equation}
	for all $i,j$. To do so, let us first assume that $ S$ is smooth. Then $\nu_{\ast}\wh\O_S=\O_S$ and therefore part 2 guarantees that by applying $\nu_\ast$, the map \eqref{eq:deg-map-relHTSS} corresponds to a morphism of analytic vector bundles
	\[ R^if_{\ast}\wtOm^j_{ X| S}\to R^{i+2}f_{\ast}\wtOm^{j-1}_{ X| S}.\]
	To see that this vanishes, it suffices to check the vanishing on fibres because $ S$ is Jacobson. But by part 4 and the aforementioned case of $S=\Spa(C,C^+)$ due to Bhatt--Morrow--Scholze, we have degeneration in every fibre. This shows the degeneration over $ S$.
	
	Finally, we deduce the more general case of reduced $S$ from this: Since we know already that \eqref{eq:deg-map-relHTSS} is a morphism of vector bundles on $ S_\proet$, we can check its vanishing after pullback to $ S_v$. We now use that  by \cite[Corollary 2.4.8]{Guo}, there exists a v-cover $h: S'\to  S$ by a smooth rigid space. It follows that we can check the vanishing of the morphism \Cref{eq:deg-map-relHTSS} after pullback to $ S'_v\to  S_v$. By part 4, this pullback is identified with (the pullback along $ S'_v\to  S'_\proet$ of) the transition map of the relative Hodge--Tate spectral sequence for the base-change $ X\times_{ S} S'\to  S'$. Since $ S'$ is smooth, we have already shown that this vanishes.
	\end{proof}
	\begin{Remark}
		We now explain the remark in the introduction that the Theorem holds more generally without the assumption that $S$ is reduced if we know a priori that $R^if_{\ast}\Omega_{ X| S}^j$ is locally free for all $i,j\geq 0$: Let $ S=\Spa(A)$ be any affinoid rigid space and let $ r:S^{\red}=\Spa(A/\mathfrak n)\to  S$ be its reduction, where $\mathfrak n\subseteq A$ is the nilradical. Let $ X^{\red}\to  S^{\red}$ be the base-change. 
		Since $(S_\proet,\wh \O_S)= (S^{\red}_{\proet},\wh \O_{S^\red})$ as ringed sites, the abutment of the spectral sequence \Cref{eq:rel-HT-in-5} is insensitive to the non-reduced structure. Hence we need to see that the same is true for the left hand side. To see this, we may identify $ S_\an= S^{\red}_\an$ to consider the short exact sequence
		\[ 0\to \mathfrak n\O_{ S}\to \O_{ S}\to \O_{ S^\red}\to 0.\]
		Pulling back to $ X_\an= X_{\an}^\red$ along $f$ and tensoring with  $\Omega_{ X| S}^j$, we obtain an exact sequence
		\[ 0\to f^\ast\mathfrak n\otimes_{\O_{ X}} \Omega_{ X| S}^j\to \Omega_{ X| S}^j\to \Omega_{ X^\red| S^\red}^j\to 0,\]
		where on the right we have used that $\Omega_{ X^\red| S^\red}^j=\Omega_{ X| S}^j\otimes_{\O_S}\O_{S^\red}$.
		Now we apply $Rf_{\an\ast}$ and use the projection formula to obtain a distinguished triangle
		\[ \mathfrak n\otimes Rf_{\ast}\Omega_{ X| S}^j\to Rf_{\ast}\Omega_{ X| S}^j\to Rf_{\ast}\Omega_{ X^\red| S^\red}^j\]
		When each $R^if_{\ast}\Omega_{ X| S}^j$ is locally free, then this gives rise for any $i\geq 0$ to a short exact sequence
		\[ 0\to \mathfrak n\otimes R^if_{\ast}\Omega_{ X| S}^j\to R^if_{\ast}\Omega_{ X| S}^j\to R^if_{\ast}\Omega_{ X^\red| S^\red}^j\to 0.\]
		The point is now that the morphism  $(S_{\proet},\wh\O_S)\to  (S_\an,\O_{S})$ factors through the morphism $ (S^\red_\an,\O_{S^\red})\to (S_{\an},\O_S)$. It follows that $\nu^{\ast}(\mathfrak n\otimes R^if_{\ast}\Omega_{ X| S}^j)=0$, hence the natural map
		\[ \nu^\ast R^if_{\ast}\Omega_{ X| S}^j\to \nu^\ast R^if_{\ast}\Omega_{ X^\red| S^\red}^j\]
		is an isomorphism for every $i$. Now all statements follow from the reduced case.
	\end{Remark}
	We can easily deduce a v-topological version of the Hodge--Tate sequence:
	\begin{Corollary}\label{c:HTrelv}
	Let $f: X\to  S$ be a smooth proper morphism of reduced rigid spaces over $K$. Then there is a natural spectral sequence of v-vector bundles on $S_v$
	\[E_2^{ij}=\lambda^{\ast}R^if_{\ast}\wtOm^j_{ X| S}\Rightarrow (R^{i+j}f_{v\ast}\wh\Z_p)\otimes_{\wh\Z_p} \O_S\]
	which degenerates at the $E_2$-page.
	\end{Corollary}
	\begin{proof}
	By \Cref{l:compare-proet-v-Zp}, we have $(R^{n}f_{v\ast}\wh\Z_p)\otimes_{\wh\Z_p} \O=\mu^\ast( R^{n}f_{\proet\ast}\wh{\Z}_p\otimes_{\wh{\Z}_p}\wh\O_S)$. The Corollary thus follows from applying $\mu^\ast$ to \eqref{eq:rel-HT-in-5}. 
	Alternatively, we could follow the same strategy as in the proof of \Cref{t:relHTSS}.
	\end{proof}
	
	\section{Splittings in the case of curve fibrations}\label{s:splittings}
	Let $X$ be a smooth proper rigid space $X$ over $C$. The Hodge filtration admits a canonical splitting when $X$ has a model over a $p$-adic subfield $K_0$ that is discretely valued with perfect residue field \cite{tate1967p}. But in general, there is no canonical such splitting, in contrast to the complex Hodge decomposition. Rather, a splitting is induced by the datum of a lift of $X$ along the square-zero thickening $B_{\dR}^+/\xi^2\to K$, and the splitting in the arithmetic case can be explained from this in view of the natural morphism $K_0\to B_{\dR}^+/\xi^2$:
	
	\begin{Proposition}[{\cite[Proposition 7.2.5]{Guo}}]\label{p:Guo-decomposition}
	Let $X$ be a smooth rigid space over $K$. Let  $\mathbb X$  be a  lift of $X$ to $B_{\dR}^+/\xi^2$, this always exists when $X$ is proper. Then $\mathbb X$ induces a decomposition
	\[ R\nu_{\ast}\wh\O_X=\oplus_{i=0}^n \wtOm^i_X[-i].\]
	This decomposition is functorial in pairs $(X,\mathbb X)$ of a smooth rigid space with a $B_{\dR}^+/\xi^2$-lift.
	\end{Proposition}

	It has already been observed by Hyodo \cite{Hyodo-HTimperfect} for abelian varieties that in contrast to the absolute case, the relative Hodge--Tate filtration is in general not split. In the context of \Cref{t:intro-HT-spectral-seq}, this is reflected by the fact that the pro-\'etale vector bundle  $(R^{1}f_{\proet\ast}\wh\Z_p)\otimes_{\wh\Z_p} \wh\O$ is usually not analytic. Indeed, there is a beautiful explanation of this phenomenon in terms of the \mbox{$p$-adic} Simpson correspondence, which relates pro-\'etale vector bundles to Higgs bundles:  
	
	When $X$ has a model over a discretely valued $p$-adic subfield $K_0$, the pro-\'etale vector bundle $(R^{n}f_{\proet\ast}\wh\Z_p)\otimes_{\wh\Z_p} \wh\O$ in \Cref{t:intro-HT-spectral-seq} can be described in terms of ``relative Higgs cohomology'' via the local $p$-adic Simpson correspondence of  Liu--Zhu \cite[Theorem~2.1.5]{LiuZhu_RiemannHilbert}.
	
	 In the algebraic case, closely related results have been obtained independently by Abbes--Gros \cite[Corollaire 6.5.34]{AbbesGrosSimpsonII}: They show that the underlying vector bundle of the Higgs bundle associated to $(R^{n}f_{\proet\ast}\wh\Z_p)\otimes_{\wh\Z_p} \wh\O$ is indeed given by the direct sum of the $E_2^{i,j}$ in \Cref{t:intro-HT-spectral-seq}, but there is a non-trivial Higgs field given by Kodaira--Spencer maps. 
	
	We expect this description to generalise to our rigid analytic setting. However, exactly as in the absolute setup, there will be an additional choice of a $B_{\dR}^+/\xi^2$-lift entering the picture. We now treat the case of relative dimension 1, which illustrates nicely the phenomenon, and is already useful for applications to the spectral curve in non-abelian $p$-adic Hodge theory:
	
	\begin{Definition}
	Let $ S$ be a rigid space, then a relative curve $ C\to S$ is a proper morphism of rigid spaces $f: C\to  S$ of pure relative dimension $1$ with connected geometric fibres.
	\end{Definition}
	Let us assume that $S$ is reduced and connected.
	When $f:X\to S$ is a smooth relative curve, \Cref{t:relHTSS} amounts to  the data of an isomorphism $\nu^\ast R^1f_{\ast}\Omega^1_{X|S}=R^2f_{\proet\ast}\wh\Z_p\otimes \wh \O_S=\wh\O_S$ and a left-exact sequence of vector bundles on $S_\proet$
	\begin{equation}\label{eq:rel-HT-seq}
	0\to \nu^\ast R^1f_{\ast}\O\to R^1f_{\proet\ast}\wh\Z_p\otimes_{\wh\Z_p}\wh\O_S\xrightarrow{\HT} \nu^\ast f_{\ast}\wtOm_{ X| S}\to 0.
	\end{equation}
	There is by \Cref{p:Deligne-constant-fibre-dim} a well-defined genus $g=\rk_{\O_S} R^1f_{\ast}\O_X$ of $f$. The outer terms are analytic vector bundles of rank $g$ while the middle is a pro-\'etale vector bundle of rank $2g$.

	\begin{Definition}\label{d:KS}
		For a smooth relative curve $f: X\to  S$, consider the cotangent sequence
	\[0\to f^{\ast}\wtOm_{ S|K}\to \wtOm_{ X|K}\to \wtOm_{ X| S}\to 0.\]
	Following \cite[\S1.1]{Katz-pcurvature}, the Kodaira--Spencer map is the boundary map of $Rf_{\ast}$ on this:
	\[ \mathrm{KS}:f_{\ast}\wtOm_{ X| S}\to R^1f_{\ast}f^{\ast}\wtOm_{ S|K}=R^1f_{\ast}\O_X\otimes \wtOm_{ S|K}\]
	\end{Definition}
	
	\begin{Proposition}\label{p:rel-HT-splitting}
	Let $ S$ be a smooth rigid space and 
	let $f: X\to  S$ be a smooth relative curve. Then the maximal exact subsequence of \cref{eq:rel-HT-seq} that consists  of analytic coherent sheaves is given by the pullback to $\ker(\KS)\subseteq f_{\ast}\wtOm_{X|S}$. More precisely, $\im(\nu_{\ast}\HT)=\ker(\KS)$.
	\end{Proposition}
	
	\begin{Remark}
	When $f$ has a model over the $p$-adic subfield $K_0$, this Proposition has a very nice explanation in terms of Liu--Zhu's $p$-adic Simpson functor \cite[Theorem~2.1.(v)]{LiuZhu_RiemannHilbert}:
	This sends $V:=R^1f_{\proet\ast}\wh\Z_p\otimes_{\wh\Z_p}\wh\O_S$ to the Higgs bundle given by $R^1f_{\ast}\O\oplus f_{\ast}\wtOm_{ X| S}$ endowed with the nilpotent Higgs field defined by $\KS$. As a pro-\'etale vector bundle is analytic if and only if the associated Higgs field vanishes, this shows that $V$ is analytic precisely over $\ker(\KS)$.
	\end{Remark}
	
	\begin{proof}[Proof of \cref{p:rel-HT-splitting}]
	The maximal analytic subsequence is given by the image of the counit map $\nu^{\ast}\nu_{\ast}\to \id$ applied to the exact sequence. It thus suffices to prove the statement about $\im(\nu_{\ast}\HT)$. 
	For this we use the short exact sequence of \textit{analytic} vector bundles on $ S_{\an}$
	\begin{equation}\label{eq:rel-HT-seq-small}
		0\to R^1f_{\an\ast}\O\to R^1(\nu\circ f_{\proet})_{\ast}\wh\O_X\xrightarrow{\HT} f_{\an\ast}\wtOm_{ X|K}\to 0
	\end{equation}
	obtained from the Leray sequence for $f_{\an}\circ \nu$, using that $R^j\nu_{\ast}\wh\O_X=\wtOm^j_X$. Here the right-exactness follows from the fact that $R^2f_{\an\ast}\O=0$ by \cref{c:vanishing-in-degree>dim}.
	We now compare this to the relative Hodge--Tate sequence via the commutative diagram on $ S_{\an}$ with exact rows:
	\[\begin{tikzcd}
		R^1(\nu\circ f_\proet)_{\ast}\wh\O_X \arrow[r, '] \arrow[d,two heads] & \nu_{\ast}R^1f_{\proet\ast}\wh\O_X \arrow[d,  "\nu_{\ast}\HT"'] \arrow[r, "\delta"] & R^2\nu_{\ast}\wh\O_S\arrow[r,"\iota"] &R^2(\nu\circ f_\proet)_{\ast}\wh\O_X\\
		f_{\an\ast}\wtOm_{X} \arrow[r, '] & f_{\an\ast}\wtOm_{X| S} \arrow[r,  "\KS"]  & R^1f_{\ast}\wtOm_S&
	\end{tikzcd}
	\]
	Here the first row comes from the 5-term exact sequence of the Leray sequence for $\nu\circ f_{\proet}$, the second column comes from  applying $R\nu_{\ast}$ to \eqref{eq:rel-HT-seq}, and the bottom row is from \cref{d:KS}.
	To see that $\im(\nu_{\ast}\HT)=\ker(\KS)$, it thus suffices to show that $\delta=0$. For this it suffices to show that $\iota$ is injective: This is the sheafification of the map sending $U\in S_{\an}$ to
	$H^2_{\proet}(U,\wh\O_S)\to H^2_{\proet}( X_{U},\wh\O_{X})$.
	This is clearly injective because $\wtOm^2_S\to f_{\ast}\wtOm_{X}^2$ is. 
	\end{proof}
	
	In this sense, the following statement about splittings is therefore ``optimal'':
	
	\begin{Proposition}
		Let $S$ be a smooth rigid space and let $f: X\to  S$ be a smooth relative curve fibration. Then any lift $\mathbb X\to \mathbb S$ of $f$ to  $B_{\dR}^+/\xi^2$ induces a splitting of \Cref{eq:rel-HT-seq} over the subsheaf $\ker \mathrm{KS}\subseteq f_{\ast}\wtOm_{ X| S}$, functorial in the lift.
	\end{Proposition}
	\begin{proof}
	For $U\in  S_{\an}$, the restriction $ X_U\to U$ admits a natural lift given by the restriction of $\mathbb X\to \mathbb S$ to the tube over $U$. \cref{p:Guo-decomposition} thus gives a decomposition, functorial in $U\in  S_{\an}$,
	\[ H^1_\proet(X_U,\wh\O_X)=H^1_{\an}(X_U,\O)\oplus H^0(X_U,\wtOm_{X})\]
	Moreover, due to the functoriality in \cref{p:Guo-decomposition}, the decomposition induced by $\mathbb S$,
	\[ H^1_\proet(U,\wh\O_S)=H^1_{\an}(U,\O)\oplus H^0(U,\wtOm_{S})\]
	is compatible with the above via the natural pullback maps.  Since the kernel of the map
	$H^1_\proet(X_U,\wh\O_X)\to H^0(U,R^1f_{\proet\ast}\wh\O_X)$
	clearly contains  $H^1_\proet(U,\O)$, we see that the kernel of the composition
	\[ H^0( X_U,\wtOm_{X})\to H^1_\proet(X_U,\wh\O_X)\to H^0(U,R^1f_{\proet\ast}\wh\O_X)\]
	contains $H^0(U,\wtOm_{U})$. Note that $H^0(U,\wtOm_{U})=H^0(X_U,f^{\ast}\wtOm^1_{U})$ as we can see locally on $U$ where $\wtOm_{U}\cong \O_U^k$ is trivial, due to the assumption that $f$ has connected geometric fibres. Upon sheafification, this shows that 
	$f_{\ast}\wtOm_{ X}\to \nu_{\ast}R^1f_{\proet\ast}\O$
	factors through $f_{\ast}\wtOm_{X}/f_{\ast}f^{\ast}\wtOm^1_{ S}$. According to \cref{d:KS}, the image of this quotient in $f_{\ast}\wtOm_{X| S}$ is precisely $\ker(\mathrm{KS})$.
	\end{proof}
	
	\appendix
	\section{Commuting Cohomology and Completion}
	In this appendix, we collect some facts about commuting cohomology and completion. These are well-known, but we do not know a reference in the literature.
	
	 One can also prove the statements in this section using derived completions, but we have instead chosen to give an elementary presentation.
	
	Let $R$ be any ring and let $\pi\in R$ be a non-zero-divisor. We say that an $R$-module $M$ has bounded $\pi$-torsion if there is $n\in \N$ such that $M[\pi^n]=M[\pi^m]$ for all $m\in \N$.
	
	\begin{Lemma}\label{l:ses-completion}
		Let $0\to M_1\to M_2\to M_3\to 0$ be a short exact sequence of $R$-modules with bounded $\pi$-torsion. Then 
		$0\to \varprojlim_n M_1/\pi^n\to \varprojlim_n M_1/\pi^n\to \varprojlim_n M_1/\pi^n\to 0$
		is exact.
	\end{Lemma}
	\begin{proof}
		We take the limit of the inverse system over $n\in \N$ given by the long exact sequences
		\[ M_2[\pi^n]\to M_3[\pi^n]\to M_1/\pi^n\to M_2/\pi^n\to M_3/\pi^n\to 0,\]
		in which $R\lim$ vanishes for the first two terms.
	\end{proof}
	\begin{Lemma}\label{l:commute-lim-and-H}
		Let $C^\bullet$ be a complex of $R$-modules which have bounded $\pi$-torsion. Let $i\in \Z$ and assume that $H^i(C^\bullet)$ and $H^{i+1}(C^\bullet)$ have bounded $\pi$-torsion. Then
		\[\textstyle H^i(\varprojlim_nC^\bullet/\pi^n)= \varprojlim_nH^i(C^\bullet)/\pi^n.\]
	\end{Lemma}
	\begin{proof}
		Let $d^i:C^i\to C^{i+1}$ be the differential. Then the short exact sequence
		\[0\to H^{i+1}(C^\bullet)\to C^{i+1}/\im d^{i}\to \im d^{i+1}\to 0 \]
		shows that $C^{i+1}/\im d^{i}$ has bounded $\pi$-torsion. It follows from \Cref{l:ses-completion} that the sequences
		\[ 0\to \ker d^{i}\to C^i\to \im d^i\to 0 \quad\text{ and } \quad  0\to \im d^{i}\to C^{i+1}\to  C^{i+1}/\im d^{i}\to 0\]
		stay exact after completion. The same holds for $i$ replaced by $i-1$. The lemma therefore now follows by applying \Cref{l:ses-completion} to the sequence
		\[ 0\to \im d^{i-1}\to \ker d^i\to H^i(C^\bullet)\to 0.\qedhere\]
	\end{proof}
	For the following, we allow an alternative setting (put in parenthesis), where we let $K$ be a perfectoid field, $R$ is a flat $\O_K$-algebra and $\pi\in \O_K$ is a pseudo-uniformiser.
	\begin{Proposition}\label{p:commute-hotimes-and-H}
	Let $C^\bullet$ be a complex of $\pi$-torsionfree $R$-modules. Let $S$ be an $R$-module such that each $C^i\otimes_RS$ is still $\pi$-torsionfree. Let $i\in \Z$ and assume either of the following:
		\begin{enumerate}
			\item  The $R$-modules $H^i(C^\bullet)$ and $H^{i+1}(C^\bullet)$ have bounded $\pi$-torsion and for every $n\in \N$, the  $R/\pi^n$-module $S/\pi^n$ is (almost) flat, or:
			\item $C^\bullet$ is bounded, and for every $j\in \N$, the $R$-modules $C^j$ and $H^j(C^\bullet)$ are (almost) flat.
		\end{enumerate}
		Then the following natural map is an (almost) isomorphism:
		\[ H^i(C^\bullet\hotimes_R S)=H^i(C^\bullet)\hotimes_R S.\]
	\end{Proposition}
	Here $\otimes$ denotes the $\pi$-adically completed tensor product.
	\begin{proof}
		We consider the natural morphism of short exact sequences of (almost) $R$-modules
		\[
		\begin{tikzcd}[column sep =0.2cm]
			0 \arrow[r] & H^i(C^\bullet\otimes_RS)/\pi^n \arrow[r]                     & H^i(C^\bullet\otimes_RS/\pi^n) \arrow[r]                     & {H^{i+1}(C^\bullet\otimes_RS)[\pi^n]} \arrow[r]                     & 0 \\
			0 \arrow[r] & H^i(C^\bullet)/\pi^n\otimes_{R/\pi^n}S/\pi^n \arrow[r] \arrow[u] & H^i(C^\bullet/\pi^n)\otimes_{R/\pi^n}S/\pi^n \arrow[r] \arrow[u] & {H^{i+1}(C^\bullet)[\pi^n]\otimes_{R/\pi^n}S/\pi^n} \arrow[r] \arrow[u] & 0
		\end{tikzcd}\]
		Here the assumption that $C^\bullet\otimes_RS$ is $\pi$-torsionfree gives rise to the top exact sequence, while either assumption 1 or 2 guarantees  that the bottom row is exact. It also shows that the middle map is an isomorphism: In the case of 1, this is clear. In the case of 2, this follows from the fact that by downward induction on $i$, also $\ker$ and $\im$ of the maps $d^i:C^i\to C^{i+1}$ are flat, hence $H^i(C^\bullet\otimes_RM)=H^i(C^\bullet)\otimes_RM$ for any $R$-module $M$.
		
		It follows that the right map is surjective, which shows that $H^{i+1}(C^\bullet\otimes_RS)$ still has  bounded $\pi$-torsion. 
		Applying $\varprojlim_n$, it follows that the two terms on the right vanish, hence
		\[\textstyle\varprojlim_n H^i(C^\bullet\otimes_RS)/\pi^n=\varprojlim_n H^i(C^\bullet)/\pi^n\otimes_{R/\pi^n}S/\pi^n=H^i(C^\bullet)\hotimes_RS.\]
		It remains to observe that since $C^\bullet\otimes_RS$ is $\pi$-torsionfree, \Cref{l:commute-lim-and-H} shows that
		\[\textstyle\varprojlim_n H^i(C^\bullet\otimes_RS)/\pi^n=H^i(C^\bullet\hotimes_RS).\qedhere\]
	\end{proof}


\begin{thebibliography}{Mum70}
		
		\bibitem[Abb10]{AbbesEGR}
		A.~Abbes.
		\newblock {\em \'{E}l\'{e}ments de g\'{e}om\'{e}trie rigide. {V}olume {I}},
		volume 286 of {\em Progress in Mathematics}.
		\newblock Birkh\"{a}user/Springer Basel AG, Basel, 2010.
		\newblock Construction et \'{e}tude g\'{e}om\'{e}trique des espaces rigides.
		
		\bibitem[ACG11]{ArbarelloGeometryofcurves}
		E.~Arbarello, M.~Cornalba,  P.~A. Griffiths.
		\newblock {\em Geometry of algebraic curves. {V}olume {II}}, volume 268 of {\em
			Grundlehren der mathematischen Wissenschaften}.
		\newblock Springer, Heidelberg, 2011.
		
		\bibitem[AG22a]{AbbesGrosSimpsonII}
		A.~Abbes,  M.~Gros.
		\newblock Correspondance de {S}impson p-adique {II}: fonctorialit\'e par image
		directe propre et syst\`emes locaux de {H}odge-{T}ate.
		\newblock {\em Preprint, arXiv:2210.10580}, 2022.
		
		\bibitem[AG22b]{AbbesGros_relativeHodgeTate}
		A.~Abbes,  M.~Gros.
		\newblock Les suites spectrales de {H}odge--{T}ate.
		\newblock {\em Preprint, arXiv:2003.04714}, 2022.
		
		\bibitem[BL93]{BLR-II}
		S.~Bosch,  W.~L\"{u}tkebohmert.
		\newblock Formal and rigid geometry. {II}. {F}lattening techniques.
		\newblock {\em Math. Ann.}, 296(3):403--429, 1993.
		
		\bibitem[BMS18]{BMS}
		B.~Bhatt, M.~Morrow,  P.~Scholze.
		\newblock {I}ntegral {$p$}-adic {H}odge theory.
		\newblock {\em Publ. Math. Inst. Hautes \'{E}tudes Sci.}, 128:219--397, 2018.
		
		\bibitem[Bos14]{Bosch-lectures}
		S.~Bosch.
		\newblock {\em {L}ectures on formal and rigid geometry}, volume 2105 of {\em
			Lecture Notes in Mathematics}.
		\newblock Springer, Cham, 2014.
		
		\bibitem[Con06]{Conrad_Ampleness}
		B.~Conrad.
		\newblock Relative ampleness in rigid geometry.
		\newblock {\em Ann. Inst. Fourier (Grenoble)}, 56(4):1049--1126, 2006.
		
		\bibitem[CS17]{CaraianiScholze}
		A.~Caraiani,  P.~Scholze.
		\newblock {O}n the generic part of the cohomology of compact unitary {S}himura
		varieties.
		\newblock {\em Ann. of Math. (2)}, 186(3):649--766, 2017.
		
		\bibitem[Del68]{Deligne-HTdeg}
		P.~Deligne.
		\newblock Th\'{e}or\`eme de {L}efschetz et crit\`eres de
		d\'{e}g\'{e}n\'{e}rescence de suites spectrales.
		\newblock {\em Inst. Hautes \'{E}tudes Sci. Publ. Math.}, (35):259--278, 1968.
		
		\bibitem[Fal88]{faltings1988p}
		G.~Faltings.
		\newblock $p$-adic {H}odge theory.
		\newblock {\em J. Amer. Math. Soc}, 1(1):255--299, 1988.
		
		\bibitem[Fal02]{Faltingsalmostetale}
		G.~Faltings.
		\newblock Almost \'{e}tale extensions.
		\newblock Number 279, pages 185--270. 2002.
		\newblock Cohomologies $p$-adiques et applications arithm\'{e}tiques, II.
		
		\bibitem[GK22]{Gaisin-Koshikawa}
		I.~Gaisin,  T.~Koshikawa.
		\newblock Relative ${A}_{\mathrm{inf}}$-cohomology.
		\newblock {\em Preprint, arXiv:2206.07983}, 2022.
		
		\bibitem[Guo23]{Guo}
		H.~Guo.
		\newblock Hodge-{T}ate decomposition for non-smooth spaces.
		\newblock {\em J. Eur. Math. Soc. (JEMS)}, 25(4):1553--1625, 2023.
		
		\bibitem[He22]{he2022cohomological}
		T.~He.
		\newblock Cohomological {D}escent for {F}altings' $p$-adic {H}odge {T}heory and
		{A}pplications.
		\newblock {\em Preprint, arXiv:2104.12645}, 2022.
		
		\bibitem[Heu21]{heuer-diamantine-Picard}
		B.~Heuer.
		\newblock {D}iamantine {P}icard functors of rigid spaces.
		\newblock {\em Preprint, arXiv:2103.16557}, 2021.
		
		\bibitem[Heu22]{heuer-sheafified-paCS}
		B.~Heuer.
		\newblock Moduli spaces in $p$-adic non-abelian {H}odge theory.
		\newblock {\em Preprint, arXiv:2207.13819}, 2022.
		
		\bibitem[Heu23]{heuer-PCT-char-p}
		B.~Heuer.
		\newblock The {P}rimitive {C}omparison {T}heorem in characteristic $p$.
		\newblock {\em In preparation}, 2023.
		
		\bibitem[HK]{HK}
		D.~Hansen,  K.~S. Kedlaya.
		\newblock Sheafiness criteria for {H}uber rings.
		\newblock {\em Preprint, available
			\href{https://kskedlaya.org/papers/criteria.pdf}{at this link}}.
		
		\bibitem[Hub96]{huber2013etale}
		R.~Huber.
		\newblock {\em \'{E}tale cohomology of rigid analytic varieties and adic
			spaces}.
		\newblock Aspects of Mathematics, E30. Friedr. Vieweg \& Sohn, Braunschweig,
		1996.
		
		\bibitem[HWZ]{HWZ}
		B.~Heuer, A.~Werner,  M.~Zhang.
		\newblock $p$-adic {S}impson correspondences for principal bundles in abelian
		settings.
		\newblock Preprint, arXiv:2308.13456.
		
		\bibitem[HX24]{HX}
		B.~Heuer,  D.~Xu.
		\newblock $p$-adic non-abelian {H}odge theory over curves via moduli stacks.
		\newblock {\em Preprint, in preparation}, 2024.
		
		\bibitem[Hyo86]{Hyodo-HTimperfect}
		O.~Hyodo.
		\newblock On the {H}odge-{T}ate decomposition in the imperfect residue field
		case.
		\newblock {\em J. Reine Angew. Math.}, 365:97--113, 1986.
		
		\bibitem[Hyo89]{HyodoVariations}
		O.~Hyodo.
		\newblock On variation of {H}odge-{T}ate structures.
		\newblock {\em Math. Ann.}, 284(1):7--22, 1989.
		
		\bibitem[Kat72]{Katz-pcurvature}
		N.~M. Katz.
		\newblock Algebraic solutions of differential equations ({$p$}-curvature and
		the {H}odge filtration).
		\newblock {\em Invent. Math.}, 18:1--118, 1972.
		
		\bibitem[Ked17]{Kedlaya-AWS}
		K.~Kedlaya.
		\newblock Sheaves, stacks, and shtukas.
		\newblock {\em Perfectoid Spaces: Lectures from the 2017 Arizona Winter
			School}, 242, 2017.
		
		\bibitem[Kie67a]{KiehlProperMapping}
		R.~Kiehl.
		\newblock Der {E}ndlichkeitssatz f\"{u}r eigentliche {A}bbildungen in der
		nichtarchimedischen {F}unktionentheorie.
		\newblock {\em Invent. Math.}, 2:191--214, 1967.
		
		\bibitem[Kie67b]{Kiehl-Theorem-A-und-B}
		R.~Kiehl.
		\newblock {T}heorem {A} und {T}heorem {B} in der nichtarchimedischen
		{F}unktionentheorie.
		\newblock {\em Invent. Math.}, 2:256--273, 1967.
		
		\bibitem[L{\"{u}}t90]{Lutkebohmer-FARAG}
		W.~L{\"{u}}tkebohmert.
		\newblock Formal-algebraic and rigid-analytic geometry.
		\newblock {\em Math. Ann.}, 286(1-3):341--371, 1990.
		
		\bibitem[L{\"{u}}t16]{Lutkebohmert_RigidCurves}
		W.~L{\"{u}}tkebohmert.
		\newblock {\em {R}igid geometry of curves and their {J}acobians}, volume~61 of
		{\em Ergebnisse der Mathematik und ihrer Grenzgebiete. 3. Folge. A Series of
			Modern Surveys in Mathematics}.
		\newblock Springer, Cham, 2016.
		
		\bibitem[LZ17]{LiuZhu_RiemannHilbert}
		R.~Liu,  X.~Zhu.
		\newblock Rigidity and a {R}iemann-{H}ilbert correspondence for {$p$}-adic
		local systems.
		\newblock {\em Invent. Math.}, 207(1):291--343, 2017.
		
		\bibitem[Mum70]{MumfordAV}
		D.~Mumford.
		\newblock {\em {A}belian varieties}.
		\newblock Tata Institute of Fundamental Research Studies in Mathematics, No. 5.
		Published for the Tata Institute of Fundamental Research, Bombay; Oxford
		University Press, London, 1970.
		
		\bibitem[Niz08]{Niziol-semistable}
		W.~Nizio\l.
		\newblock Semistable conjecture via {$K$}-theory.
		\newblock {\em Duke Math. J.}, 141(1):151--178, 2008.
		
		\bibitem[Sch]{AnalyticGeometry}
		P.~Scholze.
		\newblock Lectures on {A}nalytic {G}eometry.
		\newblock {\em Lecture notes based on joint work with D. Clausen. Available
			\href{ https://www.math.uni-bonn.de/people/scholze/Analytic.pdf}{at this
				link}}.
		
		\bibitem[Sch12]{perfectoid-spaces}
		P.~Scholze.
		\newblock {P}erfectoid spaces.
		\newblock {\em Publ. Math. Inst. Hautes \'{E}tudes Sci.}, 116:245--313, 2012.
		
		\bibitem[Sch13a]{Scholze_p-adicHodgeForRigid}
		P.~Scholze.
		\newblock {$p$}-adic {H}odge theory for rigid-analytic varieties.
		\newblock {\em Forum Math. Pi}, 1:e1, 77, 2013.
		
		\bibitem[Sch13b]{ScholzeSurvey}
		P.~Scholze.
		\newblock {P}erfectoid spaces: {A} survey.
		\newblock In {\em Current developments in mathematics 2012}, pages 193--227.
		Int. Press, Somerville, MA, 2013.
		
		\bibitem[Sch18]{Sch18}
		P.~Scholze.
		\newblock {\'E}tale cohomology of diamonds.
		\newblock {\em Preprint, arXiv:1709.07343}, 2018.
		
		\bibitem[Sta]{StacksProject}
		{The Stacks Project Authors}.
		\newblock {T}he stacks project.
		\newblock 2023.
		
		\bibitem[SW20]{ScholzeBerkeleyLectureNotes}
		P.~Scholze,  J.~Weinstein.
		\newblock {\em {B}erkeley {L}ectures on $p$-adic {G}eometry}.
		\newblock Annals of Mathematics Studies. Princeton University Press, Princeton,
		NJ, 2020.
		
		\bibitem[Tat67]{tate1967p}
		J.~T. Tate.
		\newblock {$p$}-divisible groups.
		\newblock In {\em Proc.\ {C}onf.\ {L}ocal {F}ields ({D}riebergen, 1966)}, pages
		158--183. Springer, Berlin, 1967.
		
		\bibitem[Tem04]{Temkin_local-properties}
		M.~Temkin.
		\newblock On local properties of non-{A}rchimedean analytic spaces. {II}.
		\newblock {\em Israel J. Math.}, 140:1--27, 2004.
		
		\bibitem[Tsu99]{TsujiHT}
		T.~Tsuji.
		\newblock {$p$}-adic \'{e}tale cohomology and crystalline cohomology in the
		semi-stable reduction case.
		\newblock {\em Invent. Math.}, 137(2):233--411, 1999.
		
	\end{thebibliography}
\end{document}